\documentclass[a4paper,oneside,11pt]{article}

\usepackage{amsxtra,amssymb,amsmath,amsthm,amsbsy,enumerate}
\usepackage{amssymb,latexsym}
\usepackage{graphicx,amsmath,amscd}

\usepackage[french,british]{babel}
\usepackage[all]{xy}
 \usepackage[utf8]{inputenc}
 \usepackage[T1]{fontenc}
\usepackage{amsthm}

 \usepackage{xspace}
 \usepackage{latexsym}

\theoremstyle{definition}
\newtheorem{theorem}{Theorem}[section]
\newtheorem{proposition}[theorem]{Proposition}
\newtheorem{corollary}[theorem]{Corollary}

\newtheorem{problem}[theorem]{Problem}
\newtheorem{example}[theorem]{Example}

\newtheorem{remark}[theorem]{Remark}
\newtheorem{lemma}[theorem]{Lemma}

\def\id{{\rm id}}

\def\xto{\xrightarrow}
\def\toto{\rightrightarrows}
\def\action{\curvearrowright}

\def\g{{\mathfrak g}}

\def\U{{\mathcal U}}
\def\R{{\mathbb R}}

\def\xfrom{\xleftarrow}

\def\u{\underline}

\def\<{\langle}
\def\>{\rangle}

\def\from{\gets}

\newcommand{\Cour}[1]      {[\![#1]\!]}

\def\E{\mathcal{E}}

\def\rk{{\rm rk}}

\begin{document}

\title{\bf Morita equivalences of vector bundles}

\author{Matias del Hoyo \and Cristian Ortiz}

\maketitle

\abstract{
We study vector bundles over Lie groupoids, known as VB-groupoids, and their induced geometric objects over differentiable stacks. 
We establish a fundamental theorem that characterizes VB-Morita maps in terms of fiber and basic data, and use it to prove the Morita invariance of VB-cohomology, with implications to deformation cohomology of Lie groupoids and of classic geometries. 
We discuss applications of our theory to Poisson geometry, providing a new insight over Marsden-Weinstein reduction and the integration of Dirac structures.
We conclude by proving that the derived category of VB-groupoids is a Morita invariant, which leads to a notion of VB-stacks, and solves (an instance of) an open question on representations up to homotopy.
}

\tableofcontents

\section{Introduction}


Lie groupoids are a categorification of the notion of smooth manifolds. They were first studied by A. Ehresmann and their examples include Lie groups, manifolds, Lie group actions, fibrations and foliations, among others. They manage to express both global and local symmetries, and constitute a convenient unifying framework to perform equivariant differential geometry.

A Lie groupoid yields a Lie algebroid through a differentiation process that extends the classic Lie theory. Abstract Lie algebroids arise naturally and they are not always integrable. A key example of a Lie algebroid is given by the cotangent bundle $(T^*M)_\pi$ of a Poisson structure $\pi$ on $M$. When integrable, its corresponding Lie groupoid inherits a compatible symplectic structure, setting up a fruitful interaction between Poisson geometry and Lie groupoids. This has been explored by A. Weinstein and others, and made Lie groupoids a useful tool in related areas such as quantization \cite{h,wx}, Dirac structures \cite{bcwz} and generalized complex geometry \cite{bg}. 

Another remarkable aspect of Lie groupoids is that they provide a geometric setting to deal with singular spaces such as orbifolds \cite{mp} and, more generally, differentiable stacks \cite{bx}. These are stacks over manifolds that can be presented as global quotients. Abstract stacks were introduced by Grothendieck and rely on the technical theory of sites and fibered categories. From the Lie groupoid perspective, a differentiable stack is nothing but a Morita class of Lie groupoids. Under this construction, orbifolds correspond to proper and \'etale Lie groupoids up to Morita equivalence \cite{mp}. 

\medskip


One can extend geometric notions from smooth manifolds to the framework of differentiable stacks by looking at the corresponding notion over Lie groupoids. Here we are concerned with vector bundles. A definition of vector bundles over differentiable stacks can be found in \cite{bngx}. If a differentiable stack $\mathfrak{X}$ is presented by a Lie groupoid $\u G=(G\toto M)$, then a vector bundle over $\mathfrak{X}$ corresponds to a vector bundle $E\to M$ endowed with a representation of $\u G$. However, this definition does not include the tangent stack unless in the orbifold case, which is quite unsatisfactory.

Seeking a notion of vector bundles over differentiable stacks which includes the tangent stack as an example, we study vector bundles over Lie groupoids and their Morita equivalences. Here, by a vector bundle over a Lie groupoid we mean a \emph{VB-groupoid}, 
a double structure a la Mackenzie that mixes Lie groupoids and vector bundles. This concept has recently received increased attention due to its concrete
applications in Poisson geometry \cite{m2, mx} and representation theory \cite{gm}. 

Roughly speaking, a VB-groupoid is a Lie groupoid fibration $\u \Gamma \to \u G$ whose fibers are 2-vects, categorified vector spaces. 
Besides the tangent and cotangent constructions, other 	examples arise from the representations of Lie groupoids. Extending this, it is possible to set a correspondence between general VB-groupoids and 2-term representations up to homotopy \cite{ac,gm}. 
In this paper we show that VB-groupoids serve as models for vector bundle over a differentiable stack, by studying {\em VB-Morita maps}.

\medskip


After a quick review on VB-groupoids and 2-term representations up to homotopy, our first main result is {\bf Theorem \ref{thm:Grothendieck}},
where we establish an equivalence of categories between them,
upgrading the correspondence of Gracia-Saz and Mehta \cite{gm}, and obtaining the global version of the equivalence of categories proved in \cite{djo,gm2}. 

Then we discuss the notion of VB-Morita map, which are the maps yielding isomorphisms between the corresponding stacks. We prove {\bf Theorem \ref{thm:VB-morita}}, which characterizes VB-Morita maps as those that are Morita between the base and between the fibers. Under the previous dictionary, we conclude that a VB-Morita map between VB-groupoids over a fixed base corresponds to a quasi-isomorphism between the 2-term representations up to homotopy. We derive several corollaries 
out of it.

In section 4 we study cohomology, we prove {\bf Theorem \ref{thm:cohomology}}, which says that a VB-Morita map induces an isomorphism in the level of VB-cohomology. This gives a general proof of Morita invariance of the cohomology of a Lie groupoid with coefficients in a 2-term representation up to homotopy. In particular, with the adjoint representation, we get the Morita invariance of deformation cohomology of \cite{cms}.

Section 5 presents two applications of our results to symplectic geometry. First, we observe in {\bf Proposition \ref{prop:marsdenweinstein}} that Marsden-Weinstein reduction of the cotangent lift of an action can be seen as an instance of VB-Morita equivalence. Then we apply our results to Dirac geometry. We give an interpretation of Dirac structures in terms of Lie algebroids and their adjoint and coadjoint representations up to homotopy. This allows us to show {\bf Proposition \ref{prop:presymplectic}}, which gives a neat description of pre-symplectic groupoids and a simple proof of the integration of (non-twisted) Dirac structures \cite{bcwz}.

The last section 6 is about Morita invariance of VB-groupoids, one of the main motivations for our work. 
After a thorough study of the category of VB-groupoids over a fixed base, our {\bf Theorem \ref{thm:VB-stacks}} shows that its derived category is a Morita invariant. This sets solid basis for future work on 2-vector bundles over differentiable stacks, and solves the 2-term case of Morita invariance of representations up to homotopy,
a problem posed by C. Arias Abad and M. Crainic in \cite[Ex. 3.18]{ac}.



\medskip

\textbf{Acknowledgements:}
This collaboration started at IME-USP in October 2014, during a one-month visit of MdH funded by IMPA, we acknowledge both institutions for that. We thank A. Cabrera and T. Drummond for sharing with us a preliminary version of \cite{cd} and for several related conversations. 
We thank R. Loja Fernandes for suggesting the application explored in section \ref{sec:mw}.
We also thank H. Bursztyn, M. Crainic, G. Ginot and R. Mehta for useful conversations that shaped our view on the topic.
The project was partially supported by Projeto P.V.E. 88881.030367/2013-01 (CAPES/Brazil).
CO was partially supported by grant 2016/01630-6 S\~ao Paulo Research Foundation.


\section{A detour on VB-groupoids}


We review basic definitions, constructions, and examples regarding Lie groupoids, VB-groupoids and representations up to homotopy, so as to set notations and ease the reading. We recall the ideas, constructions and results that we shall use throughout the paper. Almost all this material can be found in the literature (see e.g. \cite{bcdh,gm}). Our main contribution here is to upgrade the correspondence between VB-groupoids and representations up to homotopy to an equivalence of categories.

\subsection{Basic concepts}


A {\bf Lie groupoid} $\u G=(G\toto M)$ consists of manifolds $G,M$, two surjective submersions $s,t:G\to M$, and an associative multiplication $m:G\times_M G\to G$ admiting unit $u:M\to G$ and inverses $i:G\to G$. A {\bf Lie groupoid morphism} or {\bf map} $\phi:\u G\to \u G'$ consists of smooth maps $\phi_1:G\to G'$, $\phi_0:M\to M'$ between the objects and arrows and commuting with the five structural maps. 


A {\bf VB-groupoid} $\u \Gamma=(\Gamma\toto E)$ over $\u G$ consists of a Lie groupoid morphism $\pi:\u\Gamma\to\u G$ such that the corresponding maps $\Gamma\to G$ and $E\to M$ are vector bundle projections, for which the structure maps of $\u\Gamma$ are linear (cf. \cite{bcdh,gm,m}). 
Given $\u\Gamma,\u\Gamma'$ VB-groupoids over $\u G,\u G'$, respectively, a {\bf VB-map} $(\Phi,\phi):\u\Gamma\to\u\Gamma'$ is a Lie groupoid map $\Phi:\u\Gamma\to\u\Gamma'$ covering another $\phi:\u G\to \u G'$ and such that $\Phi$ is fiberwise linear. 
$$\xymatrix{
\u\Gamma \ar[d] \ar[r]^\Phi & \u\Gamma' \ar[d] \\
\u G \ar[r]^\phi & \u G'
}$$


Given $\u\Gamma\to \u G$ a VB-groupoid, the vector bundle $C=\ker(s:\Gamma\to E)|_M$ is called the {\bf core} and the map $\partial=t|_C:C\to E$ is called the {\bf anchor}. There are canonical identifications $\ker(s:\Gamma\to E)=t^*C$ and $\ker(t:\Gamma\to E)=s^*C$.
In light of Dold-Kan correspondence, the anchor, regarded as a two-term chain complex, encodes the restriction of $\u\Gamma$ to the units $\u M=(M\toto M)$. 
A VB-map yields a morphism between the core complexes 
$$(C\xto{\partial} E)\xto{\Phi}(C'\xto{\partial'} E')$$ 
and we say that $(\Phi,\phi)$ is a {\bf quasi-isomorphism} if it yields a (fiberwise) quasi-isomorphism between the core complexes.
We say that $\u\Gamma$ is {\bf acyclic} if $\partial$ is a fiberwise isomorphism, or equivalently, if $\u \Gamma\to \u 0$ is a quasi-isomorphism. 

\begin{remark}\label{rmk:acyclic}
Acyclic VB-groupoids were called {\em of type 1} in \cite{gm}. Note that an acyclic VB-groupoid is determined by the base and the unit bundle up to isomorphism: as a vector bundle $\Gamma\cong s^* E\oplus t^*E$, and as a groupoid there is one arrow in $\u\Gamma$ between two vectors if and only if they sit over points in the same orbit (cf. \cite[Prop. 6.5]{gm}).
\end{remark}


A VB-groupoid is a special case of a Lie groupoid fibration (cf. \cite{dhf,m}). 
Given $\u\Gamma=(\Gamma\toto E)$ a VB-groupoid, a (linear) {\bf cleavage} is defined to be a linear section $\Sigma$ of the source map $s:\Gamma\to s^* E$ over $G$.  Thus, given \smash{$y\xfrom g x$} on $G$ and $e$ over $x$, $\Sigma(g,e)$ is an arrow in $\u\Gamma$ projecting on $g$ with source $e$. The cleavage is said {\bf unital} if  $\Sigma(\id,e)=\id$, and {\bf flat} if $\Sigma(h,t\Sigma(g,e))\circ \Sigma(g,e)=\Sigma(hg,e)$. Cleavages for a VB-groupoid were called horizontal lifts in \cite{gm} and other references. 
Every VB-groupoid admits a unital cleavage, which can be constructed by using a partition of unity, but they may not admit a flat one. We will often assume that our cleavages are unital.


Given $\u\Gamma$ a VB-groupoid over $\u G$, and given $\phi:\u G'\to\u G$ a Lie groupoid map, the {\bf base-change} $\phi^*\u\Gamma$ is the VB-groupoid over $\u G'$ obtained by constructing the groupoid-theoretic fiber product between the projection $\pi:\u\Gamma\to \u G$ and $\phi$ (cf. \cite[Rmk 3.2.7]{bcdh}).


\begin{example}
Let $\u G$ be a Lie groupoid.
\begin{itemize}
 \item The {\bf tangent groupoid} $\u{TG}=(TG\toto TM)$ is a VB-groupoid over $\u G$, defined by differentiating the structural maps. Its core is $A_G$, the vector bundle of the Lie algebroid of $G$, and the anchor $\partial:A_G\to TM$ is the usual anchor map of the algebroid.
 Cleavages for $\u{TG}$ are called {\em Ehresmann connections} on \cite{ac} and {\em Cartan connections} on \cite{b}.
 \item The {\bf cotangent groupoid} $\u{T^*G}=(T^*G\toto A^*)$ is also a VB-groupoid over $\u G$, the definition of its structure maps is rather subtle, see \cite[\S~11.3]{m} for details. 
\end{itemize}
\end{example}

The cotangent $\u{T^*G}$ can be built out of $\u{TG}$ by a general duality construction. Given a VB-groupoid $\u\Gamma$, its {\bf dual} VB-groupoid $\u\Gamma^*=(\Gamma^*\toto C^*)$ can be defined, as a VB-groupoid over $\u G$, whose core-sequence is dual to that of $\u\Gamma$. 
More on duality of VB-groupoids can be consulted in \cite{bcdh,gm,m}.


We go now to another fundamental example, that allows us to think of VB-groupoids as generalized representations. Recall that a {\bf representation} $\u G\action E$ of a Lie groupoid $\u G=(G\toto M)$ over a vector bundle $E\to M$ can be described as a map 
$$\rho:G\times_M E\to E \qquad \rho(y\xfrom g x, e)=g\cdot e$$
such that $\rho_{\id}=\id$, $\rho_h\rho_g=\rho_{hg}$, and $\rho_g:E_x\to E_y$ is linear. 
In the case on which $\rho_{\id}=\id$ holds but $\rho_h\rho_g=\rho_{hg}$ may fail we refer to it as a {\bf pseudo-representation}.

\begin{example}\label{ex:representation}
Given $\rho:\u G\action E$, the corresponding {\bf action groupoid} $G\times_M E\toto E$, with source map the projection, and target map $\rho$, is naturally a VB-groupoid over $\u G$ with trivial core. Since the source map is a fiberwise isomorphism, there is only one cleavage for the action groupoid, and it is unital and flat.
\end{example}

Conversely, if $\u\Gamma\to\u G$ is a VB-groupoid with trivial core $C=0_M$, 
then the source map induces a vector bundle isomorphism $\Gamma\cong s^*E=G\times_M E$, and the composition
$G\times_{M} E \cong \Gamma\xto t E$
yields a representation $G\action E$. 
This gives a 1-1 correspondence between (isomorphism classes of) representations $G\action E$ and VB-groupoid $\u\Gamma\to \u G$ with trivial core (cf. \cite{dh,gm}), regarding VB-groupoids as generalized representations.

\subsection{Representations up to homotopy}


The previous correspondence can be extended so as to relate general VB-groupoids with certain \emph{representations up to homotopy}, as it was done in detail in \cite{gm}.
We briefly recall the basics on representations up to homotopy. They are very relevant in the theory of Lie groupoids and algebroids, necessary to make sense, for instance, of the adjoint and coadjoint representations.

\medskip


Given $\u G=(G\toto M)$ a Lie groupoid, its {\bf nerve} $N(\u G)$ is the simplicial manifold whose $p$-simplices are strings of $p$ composable arrows $G^{(p)}=\{(g_1,\dots,g_p); s(g_i) = t(g_{i+1})\}$,
and whose {\bf face} maps $\partial_i:G^{(p)}\to G^{(p-1)}$ are defined as follows:
$$
\partial_i(g_1,\dots,g_p)=
  \begin{cases}
  (g_2,\dots,g_p) & i=0\\ 
  (g_1,\dots,g_{i}g_{i+1},\dots,g_p) & i=1,\dots,p-1\\ 
  (g_1,\dots,g_{p-1}) & i=p\end{cases}$$
Out of the nerve one builds the {\bf differential graded algebra} $C(\u G)=(C^{\infty}(G^{(p)}),\delta)$,
that we should think of as the algebra of functions over the Lie groupoid:
$$\delta:C^{\infty}(G^{(p)})\to C^{\infty}(G^{(p+1)}) \qquad 
\delta\alpha=\sum_{i=0}^{p}(-1)^i\partial_i^*(\alpha),$$


Given $\u G$, a vector bundle $E\to M$ leads to a right $C(\u G)$-module $C^{\bullet}(\u G,E)=\oplus_{p\geq 0}C^{p}(\u G,E)$ by defining $C^{p}(\u G,E)=\Gamma((\pi^p_0)^*E)$, where $\pi^p_0(g_1,...,g_p)= t(g_1)$.

\begin{lemma}[\cite{ac,gm}]
There is a 1-1 correspondence between representations $\rho:\u G\action E$ and degree 1 differential operator $D:C^\bullet(\u G,E)\to C^{\bullet}(\u G,E)$ with square zero, implicit in the formula below:
\begin{align*}
D\omega(g_1,...,g_p)=&\rho_{g_1}\omega(g_2,...,g_p)+\sum_{i=1}^{p}(-1)^i\omega(g_1,...,g_ig_{i+1},...,g_p) \\ &+ (-1)^{p}\omega(g_1,...,g_{p-1}).
 \end{align*}
\end{lemma}


Following the above correspondence, if $\mathcal{E}=\oplus_{r}E_r$ is a graded vector bundle over $M$, one constructs the graded right $C^{\bullet}(\u G)$-module $C(\u G,\mathcal{E})=\oplus C(\u G,\mathcal{E})^p$, where
$C(\u G,\mathcal{E})^p=\oplus_{q-r=p}C^{q}(\u G,E_r)$, and define a \textbf{representation up to homotopy} (or ruth) $\u G\action\E$ as a total degree 1 derivation $\mathcal{D}:C(\u G,\mathcal{E})^{\bullet}\to C(\u G,\E)^{\bullet+1}$ with $\mathcal{D}^2=0$  (cf. \cite{ac}).
A {\bf morphism} $\mathcal{E}\to \mathcal{E'}$ is given by a total degree zero map $\Phi:C(\u G,\mathcal{E})\to C(\u G,\mathcal{E'})$ which is $C(\u G)$-linear and commutes with the differentials.
This way we can define the category $Rep^{\infty}(\u G)$ of representations up to homotopy of $\u G$.


Given $\u G\action\E$, the derivation $\mathcal D$ can be decomposed as a sequence $\mathcal D=\oplus_i \mathcal D_i$ by using the bigrading, and each component $\mathcal D_i$ has a simple geometric interpretation \cite{ac}: the first $\mathcal D_0=\partial$ is a differential on $\mathcal E$, the second $\mathcal D_0=\rho$ is a pseudo-representation  commuting with $\partial$, the third is a homotopy ruling the failure of $\rho$ to be multiplicative, etc.
This interpretation relies on the following correspondence:
\begin{align*}
\mathrm{Hom}^p_{C^{\bullet}(\u G)}(C^{\bullet}(\u G,B),C^{\bullet}(\u G,B'))\ni\hat\omega &\rightleftharpoons \omega\in  C^{p}(\u G,B\to B')\\
(\hat{\omega}\theta)(g_1,...,g_p,g_{p+1},...,g_{p+q}) &= \omega(g_1,...,g_p)\theta(g_{p+1},...,g_{p+q}).
\end{align*}

Here $C^{\bullet}(\u G,B\to B')$ denotes the {\bf transformation complex} associated to $\u G$. An element $\omega\in C^{p}(\u G,B\to B')$ associates to any $(g_1,...,g_p)\in G^{(p)}$ a linear map $\omega_{(g_1,...,g_p)}:B_{s(g_p)}\to B'_{t(g_1)}$. For more details, see \cite{gm}.


We refer the reader to $\cite{ac}$ for a thorough description of the general case. Here we are concerned only with {\bf 2-term representations up to homotopy}, that is, the case when $\E=E\oplus C[1]$ is concentrated on degrees 0 and 1. 

\begin{proposition}[cf. \cite{gm}]\label{prop:breaking}
A ruth $\u G\action E\oplus C[1]$ is the same as a tuple $(\partial,\rho_E,\rho_C,\gamma)$ where $\partial:C\to E$ is a linear map, $\rho^E,\rho^C$ are pseudo-representations $\u G\action C$ and $\u G\action E$, and  $\gamma\in C^2(\u G,E\to C)$ is a {\bf curvature tensor}, satisfying:
\begin{align*}\label{eq:eqsruths}
\rho^E_{g_1}\circ\partial-\partial\circ\rho^C_{g_1} &=  0 \nonumber\\
\rho^C_{g_1}\circ\rho^C_{g_2}-\rho^C_{g_1g_2}+\gamma_{g_1,g_2}\circ\partial & = 0\nonumber \\
\rho^E_{g_1}\circ\rho^E_{g_2}-\rho^E_{g_1g_2}+\partial\circ\gamma_{g_1,g_2}&= 0\nonumber \\
\rho^C_{g_1}\circ \gamma_{g_2,g_3}-\gamma_{g_1g_2,g_3}+\gamma_{g_1,g_2g_3}-\gamma_{g_1,g_2}\circ \rho^E_{g_3} &= 0,
\end{align*}
\end{proposition}
The first equation says that the quasi-actions $\rho^E,\rho^C$ commute with the differential $\partial$, the second and third say that they are multiplicative up to the curvature tensor $\gamma$, and the fourth is a compatibility condition that $\gamma$ must fulfill.


Next we extend previous result to morphisms of representations up to homotopy:

\begin{proposition}
\label{prop:ruthmap}
A morphism $\Phi:\E\to\E'$ is the same as a triple $\Phi=(\Phi_E,\Phi_C,\mu)$ where $\Phi_C:C\to C'$ and $\Phi_E:E\to E'$ and  $\mu\in C^1(\u G,E\to C')$ satisfy the following:
\begin{align*}
\partial'\circ \Phi_C-\Phi_E\circ \partial &= 0 \nonumber \\
\rho^{E'}_g\circ\Phi_E+\partial'\circ\mu_g-\Phi_E\circ\rho^E_g &= 0\nonumber \\
\Phi^C(\rho^C_g(c))-\mu_g(\partial(c))- \rho^{C'}_g\Phi^C(c)&= 0\nonumber \\
\Phi_C(\gamma_{g,h}(e))+\mu_g\rho^{E}_he+\rho^{C'}_g\mu_h(e)-\mu_{gh}(e)-\gamma'_{g,h}(\Phi_E(e))&= 0.
\end{align*}
\end{proposition}

The first equation says that $\Phi$ commutes with the anchor maps. The second and third say that $\Phi$ preserves the quasi-action of $G$ up to $\mu$. The last one should be read as a compatibility condition between $\gamma$ and $\mu$, it will become clear in Theorem \ref{thm:Grothendieck}.

\begin{proof}
Given $\Phi$, for each $p\geq 0$ we have
$$\Phi:C^{p}(\u G,E)\oplus C^{p+1}(\u G,C)\to C^{p}(\u G,E')\oplus C^{p+1}(\u G,C'),$$
and this can be decomposed into four components:
$$\begin{matrix}
C^{\bullet}(\u G,E)\xto{\widehat{\Phi_E}} C^{\bullet}(\u G,E') &
C^{\bullet}(\u G,C)\xto{\widehat{\Phi}_C} C^{\bullet}(\u G,C') \\
C^{\bullet}(\u G,E)\xto{\hat{\mu}} C^{\bullet+1}(\u G,C') &
C^{\bullet}(\u G,C)\xto{\hat{\phi}} C^{\bullet-1}(\u G,E')
\end{matrix}$$
These are $C^{\bullet}(\u G)$-linear operators corresponding to elements $\Phi_E\in \mathrm{Hom}(E,E'), \Phi_C\in\mathrm{Hom}(C,C')$ and $\mu\in C^1(\u G,E\to C')$. The last component $\phi$ vanishes by grading conventions. It is straightforward to check that the condition $\mathcal{D}'\circ \Phi=\Phi\circ \mathcal{D}$ translate into the above identities.
\end{proof}


The category of 2-term representations up to homotopy of a Lie groupoid $\u G$ is denoted by $Rep^{\infty}_{\text{2-term}}(\u G)$. We refer to \cite{ac} for more details about morphisms of representations up to homotopy of arbitrary terms and for further examples.


\subsection{Grothendieck construction}


The relation between 2-term representations up to homotopy and VB-groupoids is a linear smooth version of the classical Grothendieck correspondence between fibered categories and pseudofunctors, see \cite[\S 5]{g} or the more recent \cite[\S 3.1]{v}. 


Given a 2-term representation up to homotopy $\u G\action (E\oplus C[1])$, its {\bf Grothendieck construction} $\u \Gamma=t^* C\oplus s^* E\toto E$ is a VB-groupoid over $\u G$, defined as a semi-direct product groupoid \cite{gm}. More precisely, the structure maps of the Lie groupoid $t^*C\oplus s^*E\toto E$ are defined by
\begin{align*}
 \tilde{s}(c,g,e)=e, \quad \tilde{t}(c,g,e)=\partial(c)+\rho^E_g(e), \quad \tilde{u}_e=\left(0^C(x),u(x), e\right), \ e\in E_x \nonumber \\
(c_1,g_1,e_1)\cdot (c_2,g_2,e_2)=\left(c_1+\rho^C_{g_1}(c_2)-\gamma_{g_1,g_2}(e_2),g_1g_2,e_2\right)
\nonumber \\
(c,g,e)^{-1}=(-\rho^C_{g^{-1}}(c)+\gamma_{g^{-1},g}(e),g^{-1},\partial(c)+\rho^E_g(e)). 
\end{align*}


The main result of \cite{gm} shows that the previous construction yields a 1-1 correspondence between isomorphism classes. Our first theorem is an upgrade of that result, showing that it is in fact an equivalence of categories. This is crucial for us, for we are interested in studying the category of VB-groupoids as an invariant.

\begin{theorem}\label{thm:Grothendieck}
The Grothendieck construction is functorial, and sets an equivalence of categories
$$Rep^\infty_{\text{2-term}}(\u G)\to VB(\u G).$$
\end{theorem}

\begin{proof}
We know that this functor is essentially surjective, since, given a VB-groupoid, by picking a cleavage, we can set an isomorphism with the Grothendieck construction of a 2-term ruth (cf. \cite{gm}). We will review this later. Let us now focus in showing that  the functor is fully faithful.

Let $\u G\action\mathcal{E}=E\oplus C[1]$ and $\u G\action\mathcal{E'}=E'\oplus C'[1]$ be 2-term representations up to homotopy, and let $\Gamma=t^*C\oplus s^*E\toto E$ and $\Gamma'=t^*C'\oplus s^*E'\toto E'$ be the VB-groupoids associated to them. We have to show that there is natural bijection 
$$\mathrm{Hom}_{Rep^{\infty}_{\text{2-term}}(\u G)}(\mathcal{E},\mathcal{E'})\cong \mathrm{Hom}_{VB(\u G)}(\Gamma,\Gamma').$$

First, observe that any vector bundle map $\Gamma \to  \Gamma'$ covering $\id:G\to G$ is the same as a pair of vector bundle maps $\Phi_E:E\to E'$, $\Phi_C:C\to C'$ and an element $\mu\in C^1(\u G,E\to C')$, as it follows from the following equation:
$$\Phi(c,g,e)=(\Phi_C(c)+\mu_g(e),g,\Phi_E(e)).$$

We will show that $\Phi$ is a VB-map, i.e. it commutes with the groupoid structure maps, if and only if the components $\Phi_E,\Phi_C$ and $\mu$ define a morphism of 2-term representations up to homotopy $\mathcal{E}\to \mathcal{E'}$, i.e. they satisfy the identities in \ref{prop:ruthmap}.

From the equations
\begin{align*}
\Phi\circ t(c,g,e) &=(\Phi_E\circ \partial)(c)+\Phi_E(\rho^E_ge)\\
t\circ \Phi(c,g,e) &=(\partial'\circ\Phi_C)(c)+\partial'(\mu_g(e))+\rho^{E'}_g\Phi_E(e)
\end{align*}
it follows that $\Phi$ commutes with the target map if and only if 
the first and second conditions in \ref{prop:ruthmap} hold. 

From the equations
\begin{align*}
\Phi((c,g^{-1},0)^{-1}) & = \Phi(-\rho^C_g,g,\partial(c))=(-\Phi_C\rho^C_gc+\mu_g\partial(c),g,\Phi_E\partial(c))\\
(\Phi(c,g^{-1},0))^{-1} &= (\Phi_C(c),g^{-1},0)^{-1}=(-\rho^C_g\Phi_C(c),g,\partial'\Phi_C(c))
\end{align*}
we conclude that $\Phi$ commutes with inversion if and only if 
the first and third equations in \ref{prop:ruthmap} hold.

Finally, from the equations
\begin{align*}
\Phi((c,g,e)(\tilde{c}.h,\tilde{e})) = &\Phi(c+\rho^C_g\tilde{c}-\gamma_{g,h}\tilde{e},gh,\tilde{e})\\
=& (\Phi_C(c)+\Phi_C(\rho^C_g\tilde{c})-\Phi_C(\gamma_{g,h}\tilde{e})+\mu_{gh}\tilde{e},gh,\Phi_E(\tilde{e}))
\end{align*}
and
\begin{align*}
\Phi(c,g,e)\Phi(\tilde{c},h,\tilde{e})=&(\Phi_C(c)+\mu_ge,g,\Phi_E(e))(\Phi_C(\tilde{c})+\mu_h\tilde{e},h,\Phi_E(\tilde{e}))\\
=& (\Phi_C(c)+\mu_ge+\rho^{C'}_g\Phi_C(\tilde{c})+\rho^{C'}_g\mu_h\tilde{e}-\gamma'_{g,h}\Phi_E(\tilde{e})      ,gh,\Phi_E(\tilde{e}))
\end{align*}
we have that $\Phi$ preserves the multiplication if and only if the following equation holds:
$$\Phi_C(\rho^C_c\tilde{c})-\Phi_C(\gamma_{g,h}\tilde{e})+\mu_{gh}\tilde{e}=\mu_ge+\rho^{C'}_g\Phi_C(\tilde{c})+\rho^{C'}_g\mu_h\tilde{e}-\gamma'_{g,h}\Phi_E(\tilde{e})
$$
Since $(c,g,e)$ and $(\tilde c,h,\tilde e)$ are composable arrows, we have that $e=\partial(\tilde{c})+\rho^E_h\tilde{e}$, and we can rewrite the equation as
$$\Phi_C(\rho^C_c\tilde{c})-\Phi_C(\gamma_{g,h}\tilde{e})+\mu_{gh}\tilde{e}=\mu_g(\partial(\tilde{c})+\rho^E_h\tilde{e})+\rho^{C'}_g\Phi_C(\tilde{c})+\rho^{C'}_g\mu_h\tilde{e}-\gamma'_{g,h}\Phi_E(\tilde{e})
$$
Either assuming that $\Phi$ is a VB-map or that its components are a morphism of ruths, we already show that $\Phi_C(\rho^C_gc)-\rho^{C'}_g\Phi_C(c)=\mu_g\partial(c)$. Modulo this identity, we have that $\Phi$ commutes with multiplication if and only if
$$\mu_g\rho^{E}_h\tilde{e}+\rho^{C'}_g\mu_h\tilde{e}-\gamma'_{g,h}\Phi_E(\tilde{e})=\mu_{gh}\tilde{e}-\Phi_C(\gamma_{g,h}\tilde{e}),
$$
which is just the fourth and last condition of a morphism of representations up to homotopy, as seen in \ref{prop:ruthmap}. 

Summarizing, we have shown that a VB-map yields a morphism of ruths, and that a morphism of ruth yields a map $\Phi$ commuting with multiplication, inverse and target, and therefore, units and source as well. 
\end{proof}


Let us say a few words about how to break a VB-groupoid into a 2-term ruth, as promised. This was done in detail in \cite{gm}. We propose here a different approach that will be crucial for us later.
Recall that the {\bf arrow groupoid} $\u G^I$ of a Lie groupoid $\u G$ is a new Lie groupoid whose objects are the arrows of $\u G$, and where an arrow $g'\from g$ is a triple of composable arrows $(g',h,g)$ in $\u G$, viewed as a commutative square as before:
$$\begin{matrix}
\xymatrix@R=10pt{
x' \ar[d]_{g'}& x \ar[d]^{g} \ar[l]  \\
y' & y \ar[l] \ar[lu]^h}
\end{matrix} \qquad \leftrightharpoons \qquad
(y' \xfrom{g'} x') \xfrom{(g',h,g)}(y\xfrom g x)$$
The composition on $\u G^I$ is given by $(g'',h',g')(g',h,g)=(g'',h'g'h,g)$. The Lie groupoid maps $\sigma,\tau:\u G^I\to \u G$  and $\mu:\u G\to\u G^I$ corresponding to the source, target and unit are defined, at the level of arrows, by the formulas $\sigma(g',h,g)=g$, $\tau(g',h,g)=g'$, and $\mu(g)=(g,g^{-1},g)$. See \cite[\S 4.1]{dh} for further details. 

\begin{lemma}\label{lemma:cleavage-as-map}
There is a 1-1 correspondence between cleavages $\Sigma$ on $\u\Gamma$ and VB-maps $\rho:\sigma^*\u\Gamma\to\u\tau^*\Gamma$ over $\u G^I$ satisfying $\mu^*\rho=\id$.
\end{lemma}

\begin{proof}
This is subtle, but tautological, and it is essentially the same computation done in \cite[Prop. 2.2.3]{dh2}. 
Given $\Sigma$, we can define $\rho:\sigma^*\u\Gamma\to\tau^*\u\Gamma$ via the formulas
$$\rho(g,e)=t\Sigma(g,e) \qquad \rho((g',h,g),v)=\Sigma(g',e')v\Sigma(g,e)^{-1},$$
where $e\in E_x,e'\in E_{x'}$, \smash{$y\xfrom g x$} and \smash{$y'\xfrom{g'} x'$} are arrows in $G$, and \smash{$g'\xfrom{(g',h,g)}g$} is an arrow in $\u G^I$.
Since $\Sigma$ is unital the VB-map $\rho$ satisfies $\mu^*\rho=\id$.

Conversely, given $\rho$, we can define $\Sigma(g,e)=\rho((g,\id_x,\id_x),\id_e)$. This arrow must start in $e$ because $\mu^*\rho=\id$, and must project over $g$, hence $\Sigma$ is a cleavage. It is easy to check that it is unital.
Both constructions are mutually inverse.
\end{proof}


Given $\u\Gamma$ and fixing a (unital) cleavage $\Sigma$, the VB-map $\rho:\sigma^*\u\Gamma\to\tau^*\u\Gamma$ in lemma \ref{lemma:cleavage-as-map} clearly encodes pseudo-representations $\rho^E$, $\rho^C$ compatible with the anchor map $\partial:C\to E$, by defining $\rho^E_g(e)=\rho(g,e)$ and $\rho^C_g(v)=\rho(\id_g,v)$. 
Moreover, a curvature tensor $\gamma$ controlling the associativity is easily defined out of $\Sigma$ as
$$\gamma_{g,h}(e)=\Sigma(g,\rho_h(e))\Sigma(h,e)\Sigma(gh,e)^{-1}-ut\Sigma({gh},e)$$
so as to obtain an induced representation up to homotopy, fulfilling Proposition \ref{prop:breaking}. 


\section{VB-Morita maps}

We develop here the characterization of VB-Morita maps that play a central role in our paper. We provide a rapid overview to Morita maps, recalling notions and results, and refer to \cite{dhf} for more on Morita equivalences and differentiable stacks.

\medskip




From the stack perspective, a Lie groupoid is a stack endowed with a presentation, and from the groupoid viewpoint, a (differentiable) stack is the Morita class of a Lie groupoid (see eg \cite{bx,mp}). 
Morita equivalences of Lie groupoids can be defined either by principal bibundles or by fractions of Morita maps. A {\bf Morita map} $\phi:\u G\to \u G'$ is a map that is \emph{fully faithful} and \emph{essentially surjective}, in the sense that the source/target map define a good fibered product of manifolds $G =(M\times M)\times_{(M'\times M')} G'$  and that the map
$G'\times_{M'}M\to M$, \smash{$(y\xfrom{g} \phi(x),x)\mapsto y$} is a surjective submersion (\cite[2.2]{dh}).


It is easy to see that a Morita map yields isomorphisms on the isotropy groups and a homeomorphism on the orbit spaces. 
This result was improved with the following \emph{characterization of Morita maps}, that sheds light on the notion of differentiable stack, and will play a crucial role for us.
Given $\phi:\u G\to \u G'$ and given $x\in M$, we write
$$\bar\phi:M/G\to M'/G' \qquad \phi_x:(G_x\action N_x(O))\to(G'_{\phi(x)}\action N_{\phi(x)}(O))$$ 
for the induced maps between the orbit spaces and the normal representations.

\begin{proposition}[cf. \cite{dh,dhf}]\label{prop:criterion}
Let $\phi:\u G\to\u G'$ be a map. Then
\begin{enumerate}[(i)]
 \item $\phi$ is fully faithful if and only if $\bar\phi$ is injective and $\phi_x$ is an isomorphism on the isotropy and monomorphism in the normal directions, for every $x\in \tilde M$;
 \item $\phi$ is essentially surjective if and only if $\bar\phi$ is surjective and $\phi_x$ is an epimorphism for every $x\in \tilde M$; in such a case, the map $\bar\phi$ is open.
\end{enumerate}
Thus, $\phi$ is Morita if and only if $\bar\phi$ is a homeomorphism and $\phi_x$ is an isomorphism for every $x\in \tilde M$.
\end{proposition}


%

In this paper we are concerned with the study of Morita maps between VB-groupoids. A VB-map $(\Phi,\phi):\u\Gamma\to\u\Gamma'$ is {\bf VB-Morita} if $\Phi$ is Morita. We are going to apply the criterion given in Proposition \ref{prop:criterion} to the particular case of VB-groupoids, yielding a characterization of VB-Morita maps in terms of fiber and base data. We are going to analize first the case of the fibers, that are very simple Lie groupoids.


A {\bf 2-vect} $\u V$ is a category object in the category of vector spaces, or equivalently, a VB-groupoid over the point groupoid $\ast\toto\ast$. A 2-vect $\u V$ is equivalent to a 2-term chain complex $V_1\xto\partial V_0$ via the Dold-Kan correspondence. This complex agrees with the core complex studied in general. 

\begin{lemma}\label{lemma:2vects}
Given $\phi:\u V\to \u W$ a morphism of 2-vects, the following are equivalent:
\begin{enumerate}
 \item $\phi$ is a (categorical) equivalence;
 \item $\phi$ is VB-Morita;
 \item $\phi$ induces a quasi-isomorphism between the core complex.
\end{enumerate}
\end{lemma}
\begin{proof}
Clearly 1 implies 2, and 2 implies 3. Conversely, since every sequence of vector spaces splits, and since the Dold-Kan correspondence is a 2-equivalence, a quasi-isomorphism between 2-vects is the same as a map homotopic to the identity, hence 3 implies 1.
\end{proof}


Given $\u \Gamma\to \u G$ a VB-groupoid, and given $x\in M$, we denote by $\u\Gamma_x$ the {\bf fiber} of $\u\Gamma$ over $x$.
This is a 2-vect, defined as the base-change of $\u\Gamma$ along the inclusion $\u\ast\xto x\u G$.
Next we develop exact sequences relating the isotropies and normal representations of the fiber, the total groupoid and the base of a VB-groupoid.

\begin{lemma}
Given $\u\Gamma\to\u G$ a VB-groupoid, $e\in E$, $\pi(e)=x$, there is an exact sequence of isotropy groups
$$1\to \Gamma_x(e,e) \to \Gamma(e,e) \to G(x,x)$$
and, when $e=0_x$, the latter map is surjective, and the natural section $g\mapsto 0_g$ allows us to express the isotropy of the total group as a semi-direct product:
$$\Gamma(0_x,0_x)\cong  \Gamma_x(0_x,0_x)\rtimes G(x,x)$$
\end{lemma}

\begin{proof}
It is easy and left to the reader. This exact sequence of isotropy groups actually holds in any groupoid fibration.
\end{proof}


Given $G\toto M$ a Lie groupoid, the differential of the {\bf anchor map} $(t,s):G\to M\times M$ at a point $y\xfrom g x$ is ruled by the following exact sequence (cf. \cite{dh}):
$$0 \to T_gG(y,x) \to T_gG \to T_yM\times T_xM \to N_yM \to 0$$
where the last map sends $(w,v)$ to $[w]-g[v]$, here $g$ acts by the normal representation.
%

\begin{lemma}
Given $\u\Gamma\to \u G$ a VB-groupoid and given $e\in E$, there is a natural exact sequence
$$0\to T_{u(e)}\Gamma_{u(x)}(e,e)\to
T_{u(e)}\Gamma(e,e) \to T_xG(x,x) \xto\partial N_eE_x\to N_eE \to N_xM\to 0$$
and when $e=0_x$, then $\partial=0$ and the 0-section yields splittings
$$ T_{u(0_x)}\Gamma(0_x,0_x)=T_{u(0_x)}\Gamma_{u(x)}(0_x,0_x)\oplus T_xG(x,x)\qquad  N_{0_x}E = N_{0_x} E_x\oplus N_xM$$
\end{lemma}
\begin{proof}
In the next diagram of vector spaces the rows are exact, for in a submersion, the tangent to the fiber is the kernel of the differential. 
$$\xymatrix{
0 \ar[r] & T_{u(e)}\Gamma_{u(x)} \ar[r] \ar[d] & T_{u(e)}\Gamma \ar[d] \ar[r] & T_{u(x)}G \ar[r] \ar[d] & 0\\
0 \ar[r] & T_eE_x\times T_eE_x \ar[r] & T_eE\times T_eE \ar[r] & T_xM\times T_xM \ar[r] & 0
}$$
It follows from the Snake lemma that the kernels and cokernels of the vertical maps fit into a long exact sequence. Since they identify with the tangent spaces to the isotropies and the normal directions to the orbits respectively, the result follows. In the case on which $e=0_x$, the horizontal sequences split by the 0-section of the VB-groupoid.
\end{proof}


We are now ready to prove our second theorem, asserting that VB-Morita maps are exactly those maps that are Morita on the base and on the fibers. As seen in Lemma \ref{lemma:2vects}, the latter is equivalent to requiring the map between the core sequences to be a quasi-isomorphism.

\begin{theorem}\label{thm:VB-morita}
A VB-map $(\Phi,\phi):\u\Gamma\to\u\Gamma'$ is VB-Morita if and only if the map on the base $\phi:\u G\to\u G'$ is Morita and the maps on the fibers $\Phi_x$ are Morita for all $x\in M$.
\end{theorem}

\begin{proof}
Suppose that $\Phi$ is Morita.
Then $\Phi$ is an isomorphism on the isotropies and on the normal directions. 
By looking at the splitting isotropy sequence it follows that $\Phi_x,\phi$ are isomorphisms on the isotropies.
By looking at the splitting anchor sequence it follows that $\Phi_x,\phi$ are isomorphisms on the normal directions.
It follows that $\Phi_x$ is Morita for every $x$.
Regarding $\phi$, in light of the criterion \ref{prop:criterion}, it only remains to show that $\bar \phi:M/G\to M'/G'$ is bijective. It is clearly surjective because of the following diagram,
$$\xymatrix{
E/\Gamma \ar[r]^\cong \ar@{->>}[d] & E'/\Gamma' \ar@{->>}[d] \\
M/G \ar[r] & M'/G',
}$$
but injectivity is also easy, as it follows by looking at the 0-sections: $\bar 0_{E'}\circ\bar \phi=\bar \Phi\circ \bar 0_E$.

For the converse, suppose that $\Phi_x,\phi$ are Morita. 
Then they are isomorphisms on the isotropies and on the normal directions.
By looking at the anchor sequence, and by applying the 5-lemma, it follows that $\Phi$ is an isomorphism on the normal directions and induces infinitesimal isomorphisms on the isotropies. 
By looking at the isotropy sequence, and by applying the 4-lemma, it follows that $\Phi$ is injective on isotropies. 
Thus, again by criterion \ref{prop:criterion}, it remains to show that (i) $\Phi$ is surjective on isotropies, and (ii) $\Phi$ induces a bijection between the orbit spaces.

To show (i), let $e\in E$, and write $g=\pi(e)$, $e'=\Phi(e)$, $g'=\phi(g)=\pi(e')$. Given $v'\in\Gamma'(e',e')$ we seek for $v\in\Gamma(e,e)$ such that $\Phi(v)=v'$. 
Let \smash{$\tilde e\xfrom{w} e$} be any lift of $g$ with source $e$. 
Write $\Phi(\tilde e)=\tilde e'$ and $\Phi(w)=w'$. 
Then \smash{$e'\xfrom{v'w'^{-1}}\tilde e'$} is on the fiber $\Gamma'_{x'}$, and since $\Phi_x$ is fully faithful, we can take \smash{$e\xfrom{\tilde w}\tilde e$} on $\Gamma_x$ such that $\Phi(\tilde w)=v'w'^{-1}$. We can conclude now that by taking $v=\tilde w w$.

Regarding (ii), it will follow once we show that $\Phi$ is set-theoretically fully faithful and essentially surjective. 
To prove set-theoretic fully faithfulness, let us fix a cleavage  $\Sigma$ on $\Gamma$, and use it to canonically factor any arrow in $\Gamma$ as a vertical one followed by a horizontal one. Then, given $e_1,e_2\in E$, the following commutative square
$$\xymatrix{
\ar[r]^{\Phi} 
\ar[d]_{v\mapsto v\Sigma(\pi(v),e_1)^{-1}} 
\Gamma(e_2,e_1) &
\ar[d]^{v'\mapsto v'\Phi(\Sigma(\phi^{-1}(\pi(v)),e_1))^{-1}} 
\Gamma(e'_2,e'_1) \\
\underset{g\in G(x_2,x_1)}\coprod\Gamma_{x_2}(e_2,\rho_g(e_1)) \ar[r] & 
\underset{g'\in G'(x'_2,x'_1)}\coprod\Gamma'_{x'_2}(e'_2,\Phi(\rho_{\phi^{-1}(g')}(e_1)))
}$$
the vertical arrows are bijective, and since $\Phi_x,\phi$ are fully faithful, the bottom arrow also is.
To prove set-theoretic essential surjectivity we start with an object $e'\in E'$, we can use a cleavage $\Sigma'$ to connect it with some $\tilde e'$ over an object in the image $x'=\phi(x)$, and conclude by the fiberwise essential surjectivity.
\end{proof}


We close this section by deriving several immediate corollaries that will be used in the applications.

\begin{corollary}\label{cor:quasi-iso} 
Under the Grothendieck construction (cf. \ref{thm:Grothendieck}), quasi-isomorphisms of representations up to homotopy correspond to VB-Morita maps over $\u G$.
\end{corollary}

\begin{corollary}\label{cor:VBpullback}
Let $\phi:\u G \to \u G'$ be a Morita map and $\u\Gamma'\to \u G'$ a VB-groupoid. Then the canonical bundle map $\phi^*\u\Gamma'\to \u\Gamma$ is a VB-Morita map.
\end{corollary}

\begin{corollary}\label{cor:tangent-Morita}
A map $\phi:\u G\to \u G'$ is Morita if and only if its total differential $d\phi:\u{TG}\to \u{TG'}$ is VB-Morita.
\end{corollary}

\begin{proof}
Recall that the core sequence of $\u{TG}$ is $A_G\xto\rho TM$, and over a point $x$, the kernel is the Lie algebra of the isotropy $G(x,x)$ and the cokernel is the normal direction $N_xM$.
\end{proof}

\begin{corollary}\label{coro:Morita-dual}
A map $\Phi:\u\Gamma\to \u\Gamma'$ over the identity is VB-Morita if and only if its dual $\Phi^*:\u\Gamma'^*\to\u\Gamma^*$ is so.
\end{corollary}

\begin{corollary}
The projection $\pi:\u\Gamma\to \u G$ is Morita if and only if $\u\Gamma$ is acyclic.
\end{corollary}


\section{Morita invariance of VB-cohomology}


We overview here the notion of VB-cohomology, revisiting some results from \cite{gm} and \cite{cd}, and as a first application of our previous results, we prove here the Morita invariance of VB-cohomology. This generalizes both the Morita invariance of differentiable cohomology (cf. \cite{c}) and of deformation cohomology of Lie groupoids (cf. \cite{cms}), and is a solid step towards achieving Morita invariance of ruth cohomology. 

\medskip


Given $\u G$ a Lie groupoid, its {\bf differential cohomology} $H^\bullet(\u G)$ is obtained by taking the cohomology of the differential graded algebra of functions $C(\u G)$. The differential cohomology of $\u G$ with {\bf coefficients} in a representation $H^\bullet(\u G,E)$, or more generally, with coefficients in a representation up to homotopy $H^\bullet(\u G,\E)$, is computed by using the graded module $C(\u G,E)$ (resp. $C(\u G,\E)$) and the differential given by the representation.


Given a VB-groupoid $\Gamma\toto E$ over $G\toto M$, the following two subcomplexes of $C(\u\Gamma)$ can be considered: the {\bf linear complex} $C_{lin}(\u \Gamma)$ (cf. \cite{cd}) and the {\bf VB-complex} $C_{VB}(\u \Gamma)$ (cf. \cite{gm}).
The first one consists of the fiberwise linear cochains, and the second one consists of the fiberwise linear cochains $\phi$ that are {\bf projectable}, in the sense that they satisfy
\begin{itemize}
\item[i)] $\phi(v_1,\dots,v_{p-1},0_g)=0$, and
\item[ii)] $\phi(v_1,\dots,v_p0_g)=\phi(v_1,\dots,v_p)$.
\end{itemize}
A simple computation shows that $\phi$ is projectable if and only if 
the following hold:
\begin{itemize}
\item[i)] $\phi(v_1,\dots,v_{p-1},0_g)=0$, and
\item[ii')] $\delta(\phi)(v_1,\dots,v_p,0_g)=0$.
\end{itemize}
The {\bf linear cohomology} $H^\bullet_{lin}(\u\Gamma)$ and the {\bf VB-cohomology} $H^\bullet_{VB}(\u\Gamma)$ are defined as the cohomologies of the complexes $C^\bullet_{lin}(\u\Gamma)$ and $C^\bullet_{VB}(\u\Gamma)$. 


It turns out that the linear and VB cohomologies are isomorphic. This is proven in \cite[Lemma 3.1]{cd}. We review that proof here in a rather conceptual version, based on discussions with the authors of that article.

\begin{proposition}
The inclusion yields an isomorphism $H_{VB}(\u \Gamma)\xto\cong H_{lin}(\u \Gamma)$.
\end{proposition}

\begin{proof}
Consider the following increasing filtration on $C_{lin}(\u\Gamma)$:
$$C_{VB}(\u\Gamma)=F_1\subset F_2 \subset \dots\subset C_{lin}(\u\Gamma)$$
where $F_i$ consists of the cochains $\omega$ such that both $\omega$
and $\delta\omega$ vanish over sequences $(v_1,\dots, v_p)$ finishing on $i$ zeroes. Note that $F_i^j=C_{lin}^j(\u\Gamma)$ for $j<i$. We will show that the inclusion $F_i\subset F_{i+1}$ is a quasi-isomorphism for all $i$ and therefore $$H^k_{VB}(\u\Gamma)=H^k(F_{1})=H^k(F_{k+1})=H^k_{lin}(\u\Gamma).$$

Let $\Sigma$ be a cleavage in $\u\Gamma$. 
Given $(v_1,\dots,v_p)\in\Gamma^{(p)}$, denote by $\sigma(v_1,\dots,v_p)$ the unique arrow in $\Sigma$ that has source $s(v_p)$ and projection $\pi(v_1\dots v_p)$. We can then define a homotopy operator by
$$h:C^p_{lin}(\u\Gamma)\to C^{p-1}_{lin}(\u\Gamma)\qquad h(\phi)(v_1,\dots,v_p)=\phi(v_1,\dots,v_p,\sigma(v_1,\dots,v_p)^{-1})$$
Note that $h$ preserves the filtration, namely $h(F_i)\subset F_i$.
We claim that the chain map $I=\id+(-1)^p(h\delta-\delta h)$ is a homotopy inverse for the inclusion $F_{i-1}\subset F_i$. We must then show that $I(F_i)\subset F_{i-1}$. 

When computing $I(\phi)$ several cancellations occur, and we get
\begin{align*}
I(\phi)(v_1,&\dots,v_p) =
\phi(v_1,\dots,v_p\sigma(v_1,\dots,v_p)^{-1})+(-1)^p \phi(v_2,\dots,v_p,\sigma(v_1,\dots,v_p)^{-1})\\
&-\phi(v_1,\dots,v_{p-1},\sigma(v_1,\dots,v_{p-1})^{-1})
-(-1)^p\phi(v_2,\dots,v_{p},\sigma(v_2,\dots,v_{p})^{-1})
\end{align*}
If in addition $v_p=0_g$ for some $g$, then all the $\sigma$ arrows appearing in the above equation are $0$, and we can rewrite 
$$I(\phi)(v_1,\dots,v_p) =
(-1)^p \phi(v_2,\dots,v_p,0_{\pi(v_1\dots v_p)}^{-1})
-(-1)^p\phi(v_2,\dots,v_{p},0_{\pi(v_2\dots v_{p})}^{-1})$$
In particular, if $\phi\in F_i$, then $I(\phi)$ vanishes on sequences ending in $i-1$ zeroes, and the same holds for $\delta I(\phi)=I(\delta(\phi))$, from where the result follows.
\end{proof}


The constructions $C_{lin}(\u \Gamma), C_{VB}(\u \Gamma),H_{VB}(\u\Gamma)$ are functorial on $\u\Gamma$, for any map $\Phi:\u\Gamma\to\u\Gamma'$ must preserve linear and projectable cochains.
Our next theorem shows that VB-cohomology is a VB-Morita invariant. Its proof is a rather straightforward combination of previous results.

\begin{theorem}\label{thm:cohomology}
Let $\Phi:\u\Gamma\to \u\Gamma'$ be a VB-Morita map. Then the induced map in VB-cohomology $\Phi^*:H^{\bullet}_{VB}(\u\Gamma')\to H^{\bullet}_{VB}(\u\Gamma)$ is an isomorphism.
\end{theorem}

\begin{proof}
As shown in \cite{cd}, the natural projection $P_{\u\Gamma}:C^{\infty}(\u\Gamma)\to C^{\infty}_{lin}(\u\Gamma)$, $P_{\u\Gamma}f(\gamma):=\frac{d}{dt}|_{t=0}f(t\gamma)$, satisfies $P_{\u\Gamma}^2=P_{\u\Gamma}$ and commutes with the groupoid differential.
If we denote by $K_{\u\Gamma}=\ker(P_{\u\Gamma})$, then we have a natural direct sum decomposition of the space of differentiable cochains, and hence another one at the level of cohomology:
$$H^\bullet (\u \Gamma)=H^\bullet_{VB}(\u\Gamma)\oplus H^\bullet(K_{\u\Gamma})$$
Now, given $\Phi:\u\Gamma\to \u\Gamma'$ VB-Morita, since the differential cohomology is Morita invariant (cf. \cite[Thm 1]{c}), we have an isomorphism
$$H^\bullet_{VB}(\u\Gamma')\oplus H^\bullet(K_{\u\Gamma'})\to H^\bullet_{VB}(\u\Gamma)\oplus H^\bullet(K_{\u\Gamma})$$
and since it has to preserve the direct sum decomposition, we can conclude that the induced map on VB-cohomology is an isomorphism as well.
\end{proof}


We will derive two corollaries out of this theorem. The first one asserts the Morita invariance of ruth cohomology. In \cite{gm} VB-cohomology is related to the cohomology with coefficients in a ruth. 

\begin{proposition}[{cf. \cite[Thm 5.6]{gm}}]
Given $\u G\action E\oplus C$ a 2-term ruth, there is a canonical isomorphism
$$C(\u G,E\oplus C)\to C_{VB}(\u\Gamma^\ast)[1]$$
between the differential complex of $\u G$ with coefficients on $E\oplus C$ and the VB-complex of the dual of the Grothendieck construction, shifted by 1.
\end{proposition}

%
%
%
%
%
%
%

Combining this with our previous theorems we get the following corollary:

\begin{corollary}
Let $\phi:\u G\to \u G'$ be a Morita map and $\mathcal{E}$ a 2-term representation up to homotopy of $\u G'$. The induced map $H^{\bullet}(\u G,\phi^*\mathcal{E})\to H^{\bullet}(\u G',\mathcal{E})$ is an isomorphism.
\end{corollary}



The second corollary presented here is an independent conceptual proof of the Morita invariance of deformation cohomology. Given $\u G$ a Lie groupoid, its {\bf deformation cohomology} $H_{def}(\u G)$ was introduced in \cite{cms}. Its differential may seem arbitrary a priori, but it turns out to be very efficient in describing important properties of the Lie groupoid. We refer the reader to \cite{cms} for the definition of the deformation complex, and using their notations, we show next how we can recover it from our framework. 

\begin{proposition}\label{prop:deformation}
The following map is an isomorphism of complexes:
$$C_{def}(\u G)\to C_{VB}(\u{T^*G}) \qquad c\mapsto \phi_c:(v_p,\dots,v_1)\mapsto\<v_1,c(\pi(v_p),\dots,\pi(v_1))\>$$
\end{proposition}

\begin{proof}
The formula clearly defines a map $C_{def}^p(\u G)\to C_{lin}^p(\u{T^*G})$ for each $p$.
Moreover, it is immediate that $\phi_c$ satisfies condition i). By definition, 
$$\phi_{\partial c}(v)=-\<v_1,c(\partial_1 g)c(\partial_0g)^{-1}\> +
\sum_{i\geq 2}\<v_1,c(\partial_i g)\>$$
and
$$\partial(\phi_c)(v)=\<v_2,c(\partial_0g)\>-\<v_2v_1,c(\partial_1g)\>+
\sum_{i\geq 2}\<v_1,c(\partial_i g)\>$$
from where, denoting $w_1=c(\partial_1 g)c(\partial_0g)^{-1}$ and $w_2=c(\partial_0g)$, we have
$$\phi_{\partial c}(v)-\partial(\phi_c)(v)=\<v_2v_1,w_2w_1\>-\<v_2,w_2\>-\<v_1,w_1\>=0$$
since the canonical pairing $\u{T^*G}\times \u{TG}\to \R_\Gamma$ is multiplicative.
This proves not only that the map commutes with the differential, but also that for arbitrary $c$ both $\phi_c$ and $\partial(\phi_c)$ satisfy i), and therefore, the map takes values in the VB-complex.
It is straightforward to check that it is bijective.
\end{proof}


The previous proposition brings some light over the very definition of the deformation complex, and at the same time, provide us with conceptual simple proofs for two of the theorems of \cite{cms}.

\begin{corollary}\cite[Thm 1.1]{cms}
A Morita equivalence yields an isomorphism on deformation cohomologies.
\end{corollary}

This is an immediate application of Corollary \ref{cor:tangent-Morita}.

\begin{corollary}\cite[Thm 1.4]{cms}
A cleavage on $\u{TG}$ yields an isomorphism between the deformation cohomology and the cohomology $H(\u G,A_G\oplus TM)$ with coefficients in the adjoint representation up to homotopy.
\end{corollary}

It is remarkable that even though we need a cleavage to relate the cohomology with coefficients in the adjoint representation with either deformation cohomology or VB-cohomology, these two theories compare canonically by just unraveling the deformation differential (cf. \ref{prop:deformation}).


\section{Applications to symplectic geometry}


In this section we deal with two applications of our results to symplectic geometry. 
The first one, due to conversations with Rui Fernandes, deals with an instance of Marsden-Weinstein reduction, and it will be continued in the forthcoming paper \cite{ow}.
The second one is about pre-symplectic groupoids, which are the global objects integrating Dirac structures. We provide a characterization of them in terms of VB-Morita maps, and derive a simple conceptual proof of the global-to-infinitesimal correspondence studied in \cite{bcwz}.

\subsection{An instance of Marsden-Weinstein reduction}
\label{sec:mw}


Let $G\action M$ be a Lie group acting on a manifold. The corresponding infinitesimal action is denoted by $\g\to \mathfrak{X}(M); u\mapsto u_M$. It is well known that there is a canonical lift to a symplectic action on the cotangent bundle $(T^*M,\omega_{can})$, namely
$$G\action T^*M\qquad g\cdot(x,\xi)=(gx,\xi\circ d_{gx}(g^{-1})).$$
This action is always {\em Hamiltonian} with moment map $\mu:T^*M\to \g^*$ given by $\mu(\xi)(u) = \xi(u_M(x))$, for every $\xi \in T^*_xM$ and $u\in\g$. Within the language of VB-groupoids, we have the following conceptualization.

\begin{proposition}
If $G\times M\toto M$ is the action groupoid, then the anchor map of the cotangent VB-groupoid $\u{T^*(G\times M)}$ is given by the moment map $\mu:T^*M\to \g^*$ of the Hamiltonian lifted action $G\action (T^*M,\omega_{can})$.
\end{proposition}


\begin{proof}
The core bundle of the tangent VB-groupoid $\u{T(G\times M)}$ is the vector bundle $\g_M=\g\times M\to M$. Its core complex is $\rho:\g_M\to TM$ where $\rho(u,x)=u_M(x)$ is the induced infinitesimal action of $\g$ on $M$. The core complex of the cotangent VB-groupoid $\u{T^*(G\times M)}$ is just the dual complex $\rho^*:T^*M\to \g^*_M$, which coincides fiberwise with the moment map $\mu:T^*M\to \g^*$.  
\end{proof}


Now, when the original action $G\action M$ is free and proper, the quotient $M/G$ inherits the structure of a smooth manifold, and we can compare its cotangent bundle with the above cotangent VB-groupoid.

\begin{proposition}\label{prop:marsdenweinstein}
The cotangent of the action groupoid and the cotangent of the quotient manifold are VB-Morita equivalent.
The Marsden-Weinstein reduction $T^*M//G$ of the Hamiltonian action $G\action (T^*M,\omega_{can})$ identifies with the cotangent bundle $T^*(M/G)$ of the orbit manifold.
\end{proposition}

\begin{proof}
Let $\pi:(G\times M\toto M)\to (M/G\toto M/G)$ be the quotient map. Since the action is free and proper this is a Morita map. Its tangent map is therefore VB-Morita (cf. \ref{cor:tangent-Morita}), and we can factor it as a fixed-based map followed by a pullback:
$$\xymatrix{
(T(G\times M)\toto TM) \ar[dr] \ar[d] & \\
(\pi_1^*T(M/G)\toto \pi_0^*T(M/G)) \ar[r] & (T(M/G)\toto T(M/G))
}$$
These three maps are not only VB-Morita, but also maps of {\em LA-groupoids}, see e.g. \cite{bcdh}.
Since the pullback of VB-groupoids preserves duality, and duality preserves VB-Morita maps (cf. \ref{coro:Morita-dual}), after dualizing, we obtain the following diagram:
$$\xymatrix{
(T^*(G\times M)\toto \g^*_M) & \\
(\pi_1^*T^*(M/G)\toto 0_{M}) \ar[r] \ar[u] & (T^*(M/G)\toto 0_{M/G})
}$$
The dual object of an LA-groupoid is a PVB-groupoid, a VB-groupoid endowed with a compatible Poisson structure. Since the dual of an algebroid map is a Poisson relation, we have that this two maps are Poisson for the canonical symplectic structures.
When we look at the anchor sequences of previous diagram, we get the following:
$$\xymatrix{
(T^*M\xto j \g^*_M) & \\
(\pi_1^*T^*(M/G)\to 0_{M})\ar[r] \ar[u] & (T^*(M/G)\to 0_{M/G})
}$$
Since the vertical map is a quasi-isomorphism, we can identify $\pi_1^*T^*(M/G)$ with the kernel of the moment map $j$. Since the horizontal map is also a quasi-isomorphism, we can identify $T^*(M/G)$ with the quotient by the action by $G$. This way we conclude.
\end{proof}

 
When the action is not free, we should still think of the groupoid $(G\times M\toto M)$ as a presentation for the orbit space, just that it is not going to be a manifold in general, but a stack. Its tangent bundle is a sort of Lie algebroid over a stack, and an analog of the previous result should say that this Lie algebroid is Morita invariant. This implies defining brackets on vector fields over stacks, which is far from straightforward. For instance, the space of sections is a Lie 2-algebra rather than a mere algebra. All this is explored in detail in the forthcoming paper \cite{ow}.


\subsection{Integration of Dirac structures}


A \textbf{Dirac structure} on $M$ \cite{courant} is a subbundle $L\subseteq TM\oplus T^*M$ which is Lagrangian with respect to the canonical pairing $\langle,\rangle$ on $TM\oplus T^*M:$ 
$$\langle(X,\alpha),(Y,\beta)\rangle=\alpha(Y)+\beta(X),$$
and involutive 
with respect to the Courant bracket $\Cour{,}$ on $\Gamma(TM\oplus T^*M)$:
$$\Cour{(X,\alpha)(Y,\beta)}=([X,Y],\mathcal{L}_X\beta -i_Y\mathrm{d}\alpha).$$ 
The Courant bracket restricted to the sections of $L$, together with the canonical projection $L\to TM$, make $L$ into a Lie algebroid over $M$. 


Pre-symplectic groupoids were introduced in \cite{bcwz} as the global counterpart of Dirac structures. A {\bf pre-symplectic groupoid} is a Lie groupoid $G\toto M$ with $\mathrm{dim}(G)=2\mathrm{dim}(M)$ equipped with a closed multiplicative 2-form $\omega\in\Omega^2(G)$ which satisfies $\ker ds(x)\cap \ker dt(x)\cap \ker \omega_x=\{0\}$ for any $x\in M$.
We will next provide a nice conceptualization of this definition by using VB-Morita maps, that clarifies the relation between Dirac structures and pre-symplectic groupoids.


Before doing this, recall that a \textbf{symplectic groupoid} is a Lie groupoid $G\toto M$ equipped with a multiplicative symplectic form $\omega\in\Omega^2(G)$. They are the global counterpart of Poisson structures. The multiplicativity of $\omega$ is equivalent to $\omega^{\#}:TG\to T^*G, X\mapsto \omega(X,\cdot)$ being a (VB-)groupoid morphism.  In this case, we can rephrase the non-degeneracy condition by requiring $\omega^{\#}:TG\to T^*G$ to be an isomorphism. In particular, symplectic groupoids have isomorphic adjoint and coadjoint representations up to homotopy. We will see that the conditions defining pre-symplectic groupoids can also be interpreted as a non-degeneracy condition on
$\omega^{\#}:TG\to T^*G$.


Let $G\toto M$ be a Lie groupoid and $\Phi:\u{T G}\to \u{T^*G}$ a VB-map over ${\u G}$. For any $x\in M$, let $\Phi_x$ denote the induced map between the fibers and 
$$(\phi_0,\phi_1):(A_x\xto{\rho_x} T_xM)\to (T^*_xM\xto{\rho_x^*} A_x^*)$$
the corresponding chain map between the tangent and cotangent complexes.

\begin{lemma}\label{lemma:dirac}
The map $\Phi$ is VB-Morita if and only if the following conditions hold:
\begin{enumerate}
\item $\dim(G)=2\dim(M)$;
\item $\ker ds(x)\cap \ker dt(x)\cap \ker(\phi_1)=\{0\}$;
\item $\ker ds(x)\cap \ker dt(x)\cap \ker(\phi_0^*)=\{0\},$
\end{enumerate}
for every $x\in M$.
\end{lemma}

\begin{proof}
Assume first that $\Phi$ is VB-Morita. This is the same as saying that the induced map of complexes 
$$(A_x\xto{\rho_x} T_xM)\to (T^*_xM\xto{\rho_x^*} A_x^*)$$
is a quasi-isomorphism for every $x\in M$. In particular, both complexes must have the same Euler characteristic, from where
$$\dim(A_x)-\dim(T_xM)=\dim(T_x^*M)-\dim(A^*_x)$$
and we have $\dim M= {\rm rk} A = \dim G -\dim M$, giving 1. Condition 2 follows from the injectivity of 
the induced map on cohomology, and condition 3 from the injectivity of the dual of the induced map in cohomology.

For the converse, note that 2. and 3. imply that $(\phi_0,\phi_1)$ is a monomorphism on degree 1 cohomology and an epimorphism on degree 0 cohomology, and thus the Euler characteristic of $A_x\xto{\rho_x} T_xM$ is greater or equal than that of $T^*_xM\xto{\rho_x^*} A_x^*$, and the equality only holds if both maps on cohomology are isomorphisms. But then 1. tells us that both Euler characteristics must agree, hence we conclude.
\end{proof}


Given $\Phi$ a VB-map as above, its dual $\Phi^*$ is again a VB-map from the tangent to the cotangent. The map $\Phi$ is called {\bf symmetric} if $\Phi=\Phi^*$, and {\bf skew-symmetric} if $\Phi=-\Phi^*$.
If $\Phi$ is either symmetric or skew-symmetric, then we can identify $\phi_1=\pm \phi_0^*$, and conditions 2 and 3 in the above proposition agree.


The previous proposition combined with Theorem \ref{thm:VB-morita} provides a characterization of Lie groupoids having quasi-isomorphic adjoint and coadjoint representations up to homotopy.
We can now conclude our characterization of pre-symplectic groupoids.

\begin{proposition}\label{prop:presymplectic}
Let $\omega\in\Omega^2(G)$ be a multiplicative closed 2-form. Then $(G,\omega)$ is a pre-symplectic groupoid if and only if the map $\omega^{\#}:TG\to T^*G$ is a VB-Morita map. 
\end{proposition}

In particular, pre-symplectic groupoids have quasi-isomorphic adjoint and coadjoint representations up to homotopy.

\begin{proof}
We can identify $\omega$ with the skew-symmetric VB-map $\omega^{\#}:TG\to T^*G$. 
The result now follows as a corollary of the previous one.
\end{proof}


The previous proposition allows us to have a neat description of pre-symplectic groupoids, and also clarifies considerably their relation with Dirac structures. 
In order to do this, let us first revisit the very notion of Dirac structures.
Given a Lie algebroid $A$ over $M$ and given $\sigma:A\to T^*M$, the induced map $(\rho,\sigma):A\to TM\oplus T^*M$ identifies with a Dirac structure if and only if $\rk A=\dim M$, $(\rho,\sigma)$ is injective, and for every $a,b\in \Gamma(A)$ the 
following hold:
\begin{itemize}
\item[i)] $\sigma(a)(\rho(b))+\sigma(b)(\rho(a))=0$,
\item[ii)] $\sigma[a,b]=\mathcal{L}_{\rho(a)}b-\mathcal{L}_{\rho(b)}a-d\sigma(a)(\rho(b))$.
\end{itemize}

A bundle map $\sigma:A\to T^*M$ satisfying i) and ii) is called an \textbf{IM-2-form} in \cite{bcwz}. It is shown in \cite[Thm 4.7]{BCO} that an IM-2-form corresponds, via the construction $\Lambda=-\sigma^*\omega_{can}$, to a 2-form $\Lambda\in\Omega^2(A)$ such that $\Lambda^\#:TA\to T^*A$ is a VB-algebroid map. By abuse of terminology, we also refer to such $\Lambda^\#$ as an IM-2-form on $A$. The conditions $\rk A=\dim M$ and $(\rho,\sigma)$ injective are easily seen to be equivalent to $\Lambda^\#$ being a quasi-isomorphism between the core complexes (same computation as in \ref{lemma:dirac}). Thus we can rephrase the notion of Dirac structure as follows:

\begin{proposition}
A Dirac structure over $M$ is the same as a Lie algebroid $A$ together with a closed IM-2-form $\Lambda\in\Omega^2(A)$ that induces a quasi-isomorphism between the tangent and cotangent complexes. 
\end{proposition}

The connection between pre-symplectic groupoids and Dirac structures becomes quite evident now. Let $L$ be a Dirac structure, and asssume that $L$, as a Lie algebroid, integrates to a source-simply connected Lie groupoid $G\toto M$. Then, since Lie's second theorem holds for VB-algebroid morphisms \cite[Prop 4.3.6]{bcdh}, and since the base is source-simply connected if and only if the total groupoid is \cite[Rmk 3.1.1]{bcdh}, the morphism $\Lambda^{\#}:TA\to T^*A$ integrates to a unique morphism of Lie groupoids $\omega^{\#}:TG\to T^*G$ which is a VB-map, necessarily induced by a closed 2-form $\omega\in\Omega^2(G)$. The VB-map $\omega^{\#}:TG\to T^*G$ is actually a VB-Morita map, since the induced map at the level of complexes $(\sigma,-\sigma^*):(A_x\to T_xM)\to (T^*_xM\to A^*_x)$ is a quasi-isomorphism, and we easily recover (the non-twisted versions of) the main results in \cite{bcwz}:

\begin{corollary}[{cf. \cite[Thms 2.2,2.4]{bcwz}}]
If a Dirac structure $L\subset TM\oplus T^*M$ integrates, as a Lie algebroid, to a source-simply connected Lie groupoid $G\toto M$, then $\u G$ inherits a unique structure of pre-symplectic groupoid. Differentiation yields a one-to-one correspondence between pre-symplectic structures on $\u G$ and Dirac structures on $M$.
\end{corollary}

Summarizing, integrable Dirac (resp. Poisson) structures provide a class of examples of Lie groupoids having quasi-isomorphic (resp. isomorphic) adjoint and coadjoint representations up to homotopy, and actually, this has to be the case if the quasi-isomorphism (resp. isomorphism) is skew-symmetric. Our approach seems to unveil some connections between {\em Dirac geometry} and the pre-symplectic groupoids of \cite{bcwz} on the one side, and {\em derived symplectic geometry} and the quasi-symplectic groupoids of \cite{x} on the other side. We will further explore this elsewhere.


\section{2-Vector bundles over stacks}


We study here the category of VB-groupoids over $\u G$ as an invariant of $\u G$. First we show than any Morita map over $\u G$ is a categorical equivalence. Even though VB-groupoids with trivial core are the same as representations and hence a Morita invariant, we show with a simple example that general VB-groupoids are not. Nevertheless, our main theorem here shows that the derived category of VB-groupoids is so, solving an instance of a problem posed in \cite{ac} about representations up to homotopy, and providing a notion of 2-vector bundles over stacks that includes the tangent construction.

\subsection{VB-groupoids over a fixed base}


Given $\phi,\psi:\u\Gamma'=(\Gamma'\toto E')\to\u\Gamma=(\Gamma\toto E)$ VB-maps over $\u G=(G\toto M)$, an {\bf isomorphism}
$\alpha:\phi\cong\psi$ over $\u G$ consists of a vector bundle map
$\alpha:E'\to\Gamma|_M$ covering $\id_M$ such that 
$s\alpha(e)=\phi(e)$, $t\alpha(e)=\psi(e)$ and for every \smash{$e'\xfrom v e$} in $\u\Gamma'$ the equation $\psi(v)\alpha(e)=\alpha(e')\phi(v)$ holds. We say that a map $\phi$ is an {\bf equivalence} if it admits a quasi-inverse, namely an inverse up to isomorphisms.


We can realize an isomorphism $\alpha:\phi\cong\psi:\u\Gamma'\to\u\Gamma$ over $\u G$ as a VB-map, 
by using the {\bf arrow VB-groupoid} $\u\Gamma^I$, a variant of the construction in \cite[4.1]{dh}. Its objects are the vertical arrows, $\u\Gamma^I_0=\Gamma|_M=C\oplus E$, and its arrows are the commutative squares between them, $\u\Gamma^I_0=t^*C\oplus\Gamma\oplus s^* C$. 
The structure maps can  be witten as follows:
$$s(c',v,c)=(c,s(v)) \qquad t(c',v,c)=(c',t(v)) \qquad
i(c',v,c)=(c,i(v),c')$$
$$u(c,e)=(c,u(e),c) \qquad m((c'',v',c'),(c',v,c))=(c'',v'v,c)$$
The core sequence of $\u\Gamma^I$ identifies canonically with $C\oplus C\to C\oplus E$, $(c',c)\mapsto(c',\partial(c))$.


There are two canonical projections $\sigma,\tau:\u\Gamma^I\to\u\Gamma$ corresponding to the source and target, and an inclusion $\mu:\u\Gamma\to\u\Gamma^I$ corresponding to the unit. The maps $\sigma,\tau$ are isomorphic through the identity map $\Gamma^I_0\to\Gamma$, and this isomorphism is {\bf universal}:

\begin{lemma}
There is a 1-1 correspondence between isomorphisms $\alpha:\phi\cong\psi:\Gamma'\to\Gamma$ over $\u G$ and VB-maps $\alpha:\Gamma'\to\Gamma^I$ such that $\sigma\alpha=\phi$ and $\tau\alpha=\psi$. 
\end{lemma}


A VB-map $\phi:\u\Gamma'\to\u\Gamma$ over $\u G$ is a {\bf fibration} if it yields an epimorphism between the cores. This is an adaptation of  the usual notion of fibration between Lie groupoids (cf. \cite{dhf,m}).
The following standard argument shows that every VB-map over $\u G$ is a fibration up to equivalence.
Given $\phi$ as before, we build the fibered product $\u\Gamma'\times_{\u\Gamma}\u\Gamma^I $ between $\phi$ and $\tau$, and consider the {\bf canonical factorization} (cf. \cite[Rmk 6.2.6]{dhf}):
$$\xymatrix{
& \u\Gamma'\times_{\u\Gamma}\u\Gamma^I \ar[dr]^{\tilde\phi=\tau\pi_2} & \\
\u\Gamma' \ar[rr]_{\phi} \ar[ru]^{\tilde\iota=(\id,\mu\phi)} & & \u\Gamma
}$$
Then $\tilde\iota$ is an equivalence, with quasi-inverse the projection $\pi_1$, and $\tilde\phi$ is a fibration.


When working with general Lie groupoids, every equivalence is a Morita map, as it easily follows from characterization \ref{prop:criterion}, but in general a Morita map need not to be an equivalence. Examples of this are discussed in \cite{dh}. The next proposition shows that within the VB framework these two notions agree. In light of \ref{thm:Grothendieck}, we can think of this as a version of \cite[Prop. 3.2.8]{ac}, though our proof is completely independent.

\begin{proposition}\label{prop:VB-Morita=equivalence}
A VB-map $\phi:\u\Gamma'\to\u\Gamma$ over $\u G$ is Morita if and only if it is an equivalence.
\end{proposition}

\begin{proof}
Given $\phi:\u\Gamma'\to\u\Gamma$ a VB-Morita map, in the above canonical factorization
$\phi=\tilde\phi\tilde\iota$, we have that $\tilde\iota$ is an equivalence, and $\tilde\phi$ is not only a fibration, but also VB-Morita, by a two-out-of-three argument.
It is enough to show that $\tilde\phi$ is an equivalence. Or in other words, we may assume that the original $\phi$ is a VB-Morita fibration.

If $\phi$ is a VB-Morita fibration, it is fiberwise an epimorphism on the cores and an isomorphism on the cohomologies, then by the 5 lemma it must induce epimorphisms $E'_x\to E_x$ and also $\Gamma'_x\to\Gamma_x$. The kernel $\u K$ of $\phi$ is then a well-defined VB-groupoid. Moreover, $\u K$ must be acyclic, as it follows from Thm \ref{thm:VB-morita} and the long exact sequence in fiberwise cohomology induced by
$$0\to \u K \to \u\Gamma' \to \u\Gamma \to 0.$$

Now let $H_0$ be any linear complement for $\u K_0\subset E'$.
Since $(t,s):\Gamma'\to E'\oplus E'$ is transverse to $H_0\oplus H_0$, we have that $H_1=t^{-1}(H_0)\cap s^{-1}(H_0)$ is a vector bundle of twice the corank of $H_0$, and $\u H=(H_1\toto H_0)$ is a well-defined VB-groupoid. By counting dimensions we conclude that $\u\Gamma'=\u H\oplus \u K$, then the restriction of $\phi$ to $\u H$ is invertible, and that an inverse for $\phi|_{\u H}$ is a quasi-inverse to $\phi$, concluding the proof.
\end{proof}


As a corollary of the proof of the previous proposition, we have the following interesting consequence, reminiscent of the notion of stable isomorphism in K-theory:

\begin{corollary}\label{cor:stable}
Two VB-groupoids $\u\Gamma,\u\Gamma'$ over $\u G$ are equivalent if and only if there are acyclic VB-groupoids $\u\Omega,\u\Omega'$ over $\u G$ such that $\u\Gamma\oplus\u\Omega$ and $\u\Gamma'\oplus\u\Omega'$ are isomorphic.
\end{corollary}

\begin{proof}
A quasi-isomorphism $\phi:\u\Gamma'\to\u\Gamma$ factors as $\u\Gamma'\xto{\tilde\iota}\u{\tilde\Gamma'}\xto{\tilde\phi}\u\Gamma$ as before, with $\tilde\iota$ an injective quasi-isomorphism and $\tilde\phi$ a surjective quasi-isomorphism. It follows that $\tilde\phi$ has a section, hence $\u{\tilde\Gamma'}\cong\u\Gamma\oplus\u\Omega$ for $\u\Omega=\ker(\tilde\phi)$. 
On the other hand, the inclusion $\iota$ has always a retraction $\pi:\u{\tilde\Gamma'}\to\u\Gamma'$ and therefore $\u{\tilde\Gamma'}\cong\u\Gamma'\oplus\u\Omega'$ with $\u\Omega'=\ker(\pi)$.
\end{proof}


\subsection{Morita invariance of VB-groupoids}


VB-groupoids over $\u G$, together with VB-maps over $\u G$, and isomorphisms of maps over $\u G$, form a 2-category. For the sake of simplicity, we will restrict our attention to the following 1-categories. 
The {\bf VB-groupoid category} $VB(\u G)$ has objects the VB-groupoids and arrows the VB-maps, and the
{\bf VB-groupoid derived category} $VB[\u G]$ has objects the VB-groupoids and arrows the isomorphism classes of VB-maps. 


As recalled before, the pullback of VB-groupoids induces a {\bf base-change} functor $\phi^*:VB(\u G)\to VB(\u G')$, see eg. \cite[Rmk 3.2.7]{bcdh}.

\begin{lemma}
Given $\phi:\u G'\to\u G$ a map of Lie groupoids, the base-change functor descends to the derived categories to give $\phi^*: VB[\u G]\to VB[\u G']$.
\end{lemma}

\begin{proof}
One way to see this is by realizing isomorphisms of maps as VB-maps into the arrow VB-groupoid $\u\Gamma^I$, and noting that there is a canonical isomorphism $\phi^*(\u\Gamma^I)\cong\phi^*(\u\Gamma)^I$ compatible with $\sigma,\tau,\mu$, hence the base-change of two isomorphic maps are isomorphic through the pullback isomorphism. 
Other way is noting that, in light of \ref{prop:VB-Morita=equivalence}, the category $VB[\u G]$ is the localization of $VB(G)$ by the VB-Morita maps, that the VB-Morita maps over $\u G$ are the fiberwise quasi-isomorphisms (Theorem \ref{thm:VB-morita}), and that the quasi-isomorphisms are stable under base-change.
\end{proof}


It is well-known that the category of representations $Rep(\u G)$ of a Lie groupoid $\u G$ is a Morita invariant, it only depends on the orbit stack $M//G$. The question of whether the VB-groupoids, as a natural extension of representations, are a Morita invariant has been open for a while, and admits two variants, depending on whether one works on the derived category. 

\begin{problem}
Are the categories $VB(\u G)$ or $VB[\u G]$ a Morita invariant?
\end{problem}


Regarding the Morita invariance of $VB(G)$ we can easily find counter-examples. Next we provide a simple example where the base-change functor along a Morita fibration is neither essentially surjective nor fully faithful.

\begin{example}
Let $\u G'=(S^1\times S^1\toto S^1)$ be the pair groupoid of the circle, $\u G=(\ast\toto\ast)$ be the one-point groupoid, and $\pi:\u G'\to \u G$ the projection, that is a Morita fibration. 
If $E\to S^1$ is a non-trivial vector bundle, eg the Mobius strip, then its pair groupoid $E\times E\toto E$ is a VB-groupoid over $\u G'$ that is not isomorphic to a base-change VB-groupoid.
If $\u\Gamma=(\R\times\R\toto\R)$ is the pair groupoid of the real line, viewed as an acyclic VB-groupoid over $\u G$, then the VB-maps $\u\Gamma\to\u\Gamma$ over $\u G$ correspond to linear maps $\R\to\R$, whereas a VB-map $\phi^*(\u\Gamma)\to\phi^*(\u\Gamma)$ over $\u G'$ correspond to a linear map $\R_{S^1}\to\R_{S^1}$, that is the same as a function $S^1\to\R$.
\end{example}


We address now the more refined question regarding the derived categories, and present our main theorem, that establishes the Morita invariance of $VB[\u G]$. 

\begin{theorem}\label{thm:VB-stacks}
If $\phi:\u{\tilde G}\to \u G$ is a Morita map, then the base-change functor $\phi^*:VB[\u G]\to VB[\u{\tilde G}]$ is an equivalence.
\end{theorem}


The particular case when $\phi$ is an equivalence can be derived from the result on previous subsection. Roughly speaking, we can push-forward VB-groupoids and VB-maps along $\phi$ by pulling back them along a quasi-inverse $\psi$ of $\phi$. The proof of the general case is way more delicate and we postpone it to the next subsection.

\begin{proposition}\label{prop:VB-stack-particular-case}
If $\phi:\u{\tilde G}\to \u{G}$ is a categorical equivalence, then the base-change functor
$\phi^*:VB[\u G]\to VB[\u{\tilde G}]$ is an equivalence.
\end{proposition}

\begin{proof}
We just need to show that isomorphic maps $\phi\cong\psi:\u{\tilde G}\to \u G$ induce isomorphic base-change functors between the homotopy categories. Now, if $\alpha:\phi\cong\psi$ is a natural isomorphism, and if $\u\Gamma$ is a VB-groupoid over $\u G$, then with the aid of a cleavage $\Sigma$ on $\u\Gamma$ we can build a map $\tilde\alpha:\phi^*\u\Gamma\to\psi^*\u\Gamma$ of VB-groupoids over $G\toto M$, by $\tilde\alpha_x=\Sigma_{\alpha(x)}$. Note that $\tilde\alpha$ need not to be invertible, but it is a fiberwise quasi-isomorphism. The map $\tilde\alpha$ depends on $\Sigma$ up to isomorphism, but when passing to the derived categories we get rid of this dependence, and moreover $\tilde\alpha$ becomes invertible by Theorem \ref{thm:VB-morita}.
\end{proof}


Theorem \ref{thm:VB-stacks} is already quite interesting in the simple case on which $\u G=(M\toto M)$ is just a manifold and $\tilde G=(\coprod_{ji} U_{ji}\toto\coprod_i U_i)$ is the Lie groupoid arising from an open cover $\{U_i\}$ of $M$. We can then interpret a VB-groupoid $\u\Gamma$ over $\tilde G$ as the data of a 2-vector bundle over each $U_i$ and a sort of cocycle up to homotopy. It follows from our result that such a cocycle can always be strictified, allowing a descent construction, and yielding a globally defined 2-vector bundle over $M$. 


We propose here an alternative viewpoint over our Theorem \ref{thm:VB-stacks}. In light of \ref{thm:VB-morita}, the localization of VB-groupoids by VB-Morita maps projects over the localization of Lie groupoids by Morita maps, which is the category of differentiable stacks. Then we could define the VB-stacks over a given stack $\mathcal X$ as the fiber of that projection. This way it is rather unclear whether a VB-stack over the orbit stack of $\u G$ can be realized as a VB-groupoid over $\u G$. Our theorem ensures that this in fact the case, that localizing and taking fibers commute.


Finally, by combining \ref{thm:Grothendieck}, \ref{cor:quasi-iso} and \ref{thm:VB-stacks}, we can give a positive answer to (an instance of) the Morita invariance of representations up to homotopy (cf. \cite[Ex. 3.18]{ac}).

\begin{corollary}
The derived category of the 2-term representations up to homotopy of a Lie groupoid is a Morita invariant.
\end{corollary}

In \cite{dht} we explore a geometric realization of higher representations up to homotopy as simplicial vector bundles over the nerve of the Lie groupoid. We expect this to be useful in extending some of the results obtained here, such as the Morita invariance, from the 2-term to the general case. 

\subsection{Proof of the main theorem}


We proceed as follows. First we show that, by a standard argument, we can suppose that $\phi$ is a {\em \v{C}ech fibration}, an equivalence given by an open cover of the unit manifold. 
Then we show that $\phi^*$ is fully faithful. Even though a map between pullback VB-groupoids may not descend a priori, we show that it does so after averaging with respect to a partition of 1, and that this averaging does not change the isomorphism type. This gives fullness, and also faithfulness, after realizing an isomorphism of VB-maps as a VB-map, in the same way a homotopy of maps is itself a map. 
Finally, we show that $\phi^*$ is essentially surjective, starting with an arbitrary VB-groupoid and replacing it with other equivalent one that admits a cleavage with flatness properties, again by using a partition of 1.


\medskip

\noindent{\bf Step 1: Restricting to \v{C}ech fibrations} 


Given a Lie groupoid $G\toto M$, and given $\U=\{U_i\}_i$ an open cover of $M$, we can build the pullback groupoid $\u{G_\U}=(\coprod_{j,i}G(U_j,U_i)\toto \coprod_i U_i)$ corresponding to the surjective submersion $\coprod_i U_i\to M$. Its structure maps are induced by those of $\u G$. The canonical projection 
$$\pi_\U:\u{G_\U}\to \u G$$
is a Morita fibration, we call it a {\bf \v{C}ech fibration}. 
The kernel is just $\coprod_{ji} U_{ji}\toto\coprod_i U_i$.


This type of fibrations are cofinal among the Morita fibrations over $\u G$. If $\phi:\u{\tilde G}\to \u G$ is any other Morita fibration, then a collection of local sections $\sigma_i:U_i\to \tilde M$ canonically induces a Lie groupoid map $\sigma:\u{G_\U}\to\u{\tilde G}$, and we get a {\bf refining} \v{C}ech fibration as follows:
$$\xymatrix{
 & \u{\tilde G} \ar[d]^\phi & \\
 \u{G_\U} \ar[ru]^\sigma \ar[r]_{\pi_\U}  & \u G
}$$


We want to show that the derived category $VB[\u G]$ is a Morita invariant of $\u G$. Given $\phi:\u{\tilde G}\to\u G$ a Morita map, we need to show that the pullback is an equivalence of categories.
In light of the canonical factorization of a Morita map as an equivalence followed by a Morita fibration (cf. \cite[Rmk 6.2.6]{dhf}), and in light of the particular case already proven (cf. Prop. \ref{prop:VB-stack-particular-case}), we can suppose that $\phi$ is a Morita fibration. And since \v{C}ech fibrations are cofinal among the Morita fibrations, by the following standard argument, we can restrict our attention to them.

\begin{lemma}
If $\phi^*$ is an equivalence of categories for every \v{C}ech fibration, then the same holds for every Morita map.
\end{lemma}

\begin{proof}
As explained, we can assume $\phi^*$ to be a fibration. So start with an arbitrary Morita fibration $\phi:\u{\tilde G}\to \u G$. Take a refining \v{C}ech fibration as above, and by using the induced open cover $\{\tilde U_i=\phi^{-1}(U_i)\}_i$, build the corresponding \v{C}ech fibration over $\u{\tilde G}$:
$$\xymatrix{
\u{\tilde G_{\tilde\U}} \ar[r]^{\tilde\pi_{\tilde\U}} \ar[d]_{\phi_\U} & \u{\tilde G} \ar[d]^\phi & \\
\u{G_\U} \ar[r]^{\pi_\U} \ar[ur]^\sigma & \u G
}$$
The map $\phi$ induces another $\phi_\U$ completing the diagram.
By hypothesis, the base-change functors $\tilde\pi_{\tilde\U}^*=\sigma^*\phi_\U^*$ and $\pi_\U=\phi^*\sigma^*$ are equivalences of categories. Then $\sigma^*$ has to be an equivalence of categories. And by a two-out-of-three argument, the original base-change functor $\phi^*$ also is.
\end{proof}


\medskip

\noindent{\bf Step 2: $\phi_\U^*$ is fully faithful}


We start by showing that it is full.
Given $\pi_\U:\u{G_\U}\to \u G$ a \v{C}ech fibration, $\u \Gamma$ and $\u \Gamma'$ VB-groupoids over $\u G$, and $\psi:\pi_\U^*(\u \Gamma)\to\pi_\U^*(\u \Gamma')$ a VB-map over $\u{G_\U}$, we want to show that there is a map $\phi:\u \Gamma\to \u \Gamma'$ such that $\pi_\U^*(\phi)\cong\psi$.
This is equivalent to build a VB-map $\phi:\u \Gamma\to \u \Gamma'$ such that the following square commutes up to homotopy, where the maps $\pi,\pi'$ are the canonical projections.
$$\xymatrix{\pi_\U^*(\u \Gamma) \ar[r]^{\pi} \ar[d]_{\psi} & \u \Gamma \ar@{-->}[d]^\phi\\ \pi_\U^*(\u \Gamma') \ar[r]^{\pi'} & \u \Gamma'}$$


We will cook up $\phi$ by using the following elementary property:

\begin{lemma}
Let $\phi:\u{\tilde G}\to \u G$ be a Morita fibration with kernel $K$. 
A map $\psi:\u{\tilde G}\to \u H$ factors through $\phi$ as $\tilde\psi\phi=\psi$ if and only if $\psi$ maps $K$ into identities.
\end{lemma}

\begin{proof}
This is a straightforward consequence of the fact that a map constant over the fibers of a surjective submersion descends to the base manifold.
\end{proof}

In our case, both $\pi,\pi'$ are \v{C}ech fibrations with kernel $K=(\coprod E_{ji}\toto\coprod E_i)$, where $q:E\to M$ is the projection. We will show that $\psi$ is isomorphic to a map $\tilde\psi$ that preserves the kernel, and therefore it descends to give a map $\phi$ as we want.


Let us a review a general construction. Given $\phi:\u\Gamma\to\u\Gamma'$ a VB-map over $\u G$ and $\alpha:E\to C'$ a linear map, the {\bf twisting} of $\phi$ by $\alpha$ is the map defined below:
$$\phi^\alpha_0(e)=\phi_0(e)+\partial(\alpha(e))\qquad
\phi_1^\alpha(e'\xfrom v e)=(\alpha(e')+u\phi_0(e'))\circ\phi(g)\circ
(\alpha(e)+u\phi_0(e))^{-1}$$
This way we have an isomorphism $\alpha:\phi\cong\phi^\alpha$, and actually, for any isomorphism $\alpha:\phi\cong\psi$ over the identity of $\u G$ we have $\psi=\phi^\alpha$. 


Coming back, we seek for a vector bundle map $\alpha:\coprod_i E_i\to \coprod_i C'_i\subset\coprod \Gamma'_i$, that amounts to be the same as a collection $\{\alpha_i:E_i\to C'_i\}_i$.
Writing $\psi_0=\coprod_i\psi_i:\coprod_i E_i\to\coprod E'_i$ and $\psi_1=\coprod_{ji}\psi_{ji}:\coprod \Gamma_{ji}\to\coprod_{ji}\Gamma'_{ji}$ for the induced maps on objects and arrows, we can define a family 
$\beta_{ji}:E_{ji}\to C'_{ji}$, $\beta_{ji}=\psi_{ji}\circ u - u\circ \psi_i$. This is the vertical obstruction for $\psi$ to preserve the kernel \smash{$K=\{(e,j)\xfrom{(u(e),j,i)}(e,i)\}$}:
$$\xymatrix@C=40pt{
 (\psi_j(e),j) & \\
 (\psi_i(e),j) \ar[u]^{(\beta_{ji}(e),j,j)} &  
 \ar[lu]_{(\psi_{ji}(u(e)),j,i)}
 \ar[l]^{(u(\psi_i(e)),j,i)}
 (\psi_i(e),i)
 }$$
This $\beta$ is a cocycle, in the sense that $\beta_{kj}(e)+\beta_{ji}(e)=\beta_{ki}(e)$ holds for any $k,j,i$.


We now integrate the cocycle $\beta$, using a partition of 1 
$\{\lambda_i\}_i$ 
subordinated to $\{U_i\}_i$, by defining
$\alpha_i:E_i\to C_i'$, $\alpha_i(e)=\sum \lambda_j(x)\beta_{ji}(e)$. 
Twisting the original $\psi$ by $\alpha$ we get the desired isomorphic map that preserves the kernel and descends to the quotient:
\begin{align*}
\psi^\alpha(u(e),j,i) 
  &=(\alpha_j(e)+u\psi_{j}(e),j,j)\circ(\psi_{ji}(u(e)),j,i)\circ(\alpha_i(e)+u\psi_{i}(e),i,i)^{-1}\\
  &=(\sum_{k}\lambda_k(x)\beta_{kj}(e)+\beta_{ji}(e) -\sum_k\lambda_k(x)\beta_{ki}(e),j,j)+ \psi^\alpha(u(e),j,i)\\
  &=(0,j,j)+\psi^\alpha(u(e),j,i)=\psi^\alpha(u(e),j,i).
\end{align*}



We have now completed the proof of fullness. As we said before, in order to show that the base-change is faithful, which means injective on isomorphism classes of maps, we realize an isomorphism as a VB-map, and use the fullness we have just established.

Let $\psi,\psi':\u \Gamma\to \u \Gamma'$ over $\u G$, such that 
$\pi_\U^*(\psi)$ and $\pi_\U^*(\psi')$ are isomorphic, namely there is a homotopy $h:\pi_\U^*(\u\Gamma)\to \pi_\U^*(\u\Gamma')^I\cong \pi_\U^*(\u\Gamma'^I)$ such that $\sigma h=\pi_\U^*(\psi)$ and $\tau h=\pi_\U^*(\psi')$. Then since the base-change functor is full we know there exists an $h':\u\Gamma\to \u\Gamma'^I$ such that $\pi_\U^*(h')=h$, and therefore, $\sigma h'\cong\psi$ and $\tau h'=\psi'$.


\medskip

\noindent{\bf Step 3: $\phi^*$ essentially surjective}

Given $\u \Gamma$ over $\u{G_\U}$, we want to find $\u{\tilde\Gamma}$ over $\u G$ and an equivalence $\u\Gamma\cong \pi_\U^*(\u{\tilde\Gamma})$.
To do this we first characterize the VB-groupoids over $\u{G_\U}$ that are a pullback through $\pi_\U$. They are those admitting an {\bf $\U$-flat} cleavage $\Sigma$, namely one that restricted to the kernel $K$ of $\pi_\U$ is flat.


\begin{lemma}
A VB-groupoid $\u \Gamma$ over $\u G_\U$ is isomorphic to a pullback VB-groupoid $\u{\tilde\Gamma_\U}$ if and only if it admits a cleavage $\Sigma$ that is $\U$-flat.
\end{lemma}

\begin{proof}
Given a VB-groupoid $\u\Gamma$ over $\u G$, and given a cleavage $\Sigma$ on it, there is an induced cleavage on the base-change $\pi_\U*(\u\Gamma)$, and this cleavage is $\U$-flat. Conversely, suppose that $\u\Gamma$ is a VB-groupoid over $\u{\tilde G}$ and that is endowed with a $\U$-flat cleavage $\Sigma$. This cleavage defines a free proper wide subgroupoid $\u K_{\u\Gamma}\subset\u\Gamma$, and therefore, a fibration with base $\u\Gamma/\u{K}_{\u\Gamma}$ (cf. \cite[Prop. 6.2.4]{dhf}) 
$$\xymatrix{
\u K_{\u\Gamma} \ar[r] & \u\Gamma \ar[r] \ar[d] & \u\Gamma/\u{K}_{\u\Gamma} \ar[d] \\
\u K \ar[r] & \u{G_\U} \ar[r] & \u G
}$$
It is straightforward to check that $\u\Gamma/\u{K}_{\u\Gamma}\to \u G$ is a VB-groupoid projection, and that $\pi_\U^*(\u\Gamma/\u K_{\u\Gamma})$ is isomorphic to $\u\Gamma$.
\end{proof}


Consider now an arbitrary VB-groupoid $\u\Gamma$ over $\u{G_\U}$.
First we will replace $\u \Gamma$ by an equivalent VB-groupoid $\u{\tilde\Gamma}=\u \Gamma\oplus\u \Omega$, where $\Omega$ is acyclic, that does admit an {\bf invertible} cleavage $\Sigma$, in the sense that the associated pseudo-representation $\rho$ is by linear isomorphisms. 
Starting with a (unital) cleavage $\Sigma$, we regard it as a VB-map
$\rho:\u{\sigma^*\Gamma}\to \u{\tau^*\Gamma}$ over $\u{G^I}$ (cf. \ref{lemma:cleavage-as-map}), and observe that this is a quasi-isomorphism, for it is fiberwise invertible up to homotopy (cf. \ref{prop:breaking}). Then by Theorem \ref{thm:VB-morita} $\rho$ is a Morita map, and by Corollary \ref{cor:stable}, we can find $\u \Omega',\u \Omega''$ acyclic and an isomorphism $\u{\sigma^*\Gamma}\oplus\u\Omega'\xto\sim\u{\tau^*\Gamma}\oplus\u\Omega''$.
We will further assume that our open cover $\U$ is {\em good}, in the sense that every finite intersection is diffeomorphic to $\R^n$. The fact that these covers are cofinal is rather standard. Then, since the unit bundle of any VB-groupoid over $\u{G_\U}$ has to be trivial, we conclude that 
$\u\Omega'\cong\u\Omega''\cong\sigma^*(\u\Omega)\cong \tau^*(\u\Omega)$
where $\u\Omega$ is the unique acyclic VB-groupoid over $\u G$ with unit bundle trivial of rank $q$ (cf. Rmk \ref{rmk:acyclic}).
Since base-change preserves direct sums, the resulting isomorphism $$\rho':\sigma^*(\u\Gamma\oplus\u\Omega)\xto\sim\tau^*(\u\Gamma\oplus\u\Omega)$$
correspond to an invertible cleavage $\Sigma'$ over $\u\Gamma\oplus\u\Omega$ (again by \ref{lemma:cleavage-as-map}).
For this we need that the isomorphism $\rho$ is trivial over the identities, namely $\mu^*\rho'=\id$, but this can be achieved by picking carefully a linear complement to the kernel when constructing the quasi-inverse (cf. \ref{prop:VB-Morita=equivalence}).


Now, starting with $\u\Gamma$ a VB-groupoid over $\u{G_\U}$ and $\Sigma$ an invertible cleavage, we can easily construct a new $\Sigma$ that is symmetric, in the sense that $\rho_{ji}=\rho_{ij}^{-1}$, $\rho$ being the induced pseudo-representation. This step is very easy. Just establish a total order on the set indexing the open cover, and define a new cleavage $\Sigma'$ by setting 
 $$\Sigma'_{ji}=\begin{cases}
                 \Sigma_{ji} & j\geq i \\
		 \Sigma_{ij}^{-1} & i\geq j
		 \end{cases}.$$
 There is no ambiguity on the definition because $\Sigma$ is unital.
 

Finally, starting with $\u\Gamma$ and $\Sigma$ invertible and symmetric, we build a new cleavage $\Sigma'$ that is $\U$-flat, by performing an averaging.
Our cleavage $\Sigma$, invertible and satisfying unital and symmetry,  induces a ruth, with associated curvature tensor $\gamma_{kji}$. We now define 
 $$\beta_{ji}(e)=\sum_r\lambda_r(x)\gamma_{jri}(e)$$
where $\lambda_r$ is a partition of 1 subordinated to the open cover.
Our new cleavage is given by
\begin{align*}
\Sigma'_{ji}(g,e)
  &=\Sigma_{j,i}(g,e) + 0_{g,j,i}\beta_{ji}(e)\\
  &=\Sigma_{j,i}(g,e) + 0_{g,j,i}\sum_r\lambda_r(x)\gamma_{jri}(e)\\
  &=\Sigma_{j,i}(g,e) + \sum_r\lambda_r(x)0_{g,j,i}\gamma_{jri}(e)\\
  &=\Sigma_{j,i}(g,e) + \sum_r\lambda_r(x)(\Sigma_{jr}\Sigma_{ri}-\Sigma_{ji})\\
  &=\sum_r\lambda_r(x)\Sigma_{jr}\Sigma_{ri}
\end{align*}
It is easy to check now that the new cleavage $\Sigma'$ is $\U$-flat:
\begin{align*}
\Sigma'_{kj}\Sigma'_{ji}
  &=(\sum_r\lambda_r\Sigma_{kr}\Sigma_{rj})(\sum_r\lambda_r\Sigma_{jr}\Sigma_{ri})\\
  &=\sum_r\lambda_r\Sigma_{kr}\Sigma_{rj}\Sigma_{jr}\Sigma_{ri}\\
  &=\sum_r\lambda_r\Sigma_{kr}\Sigma_{ri}=\Sigma'_{ki}
\end{align*}
This completes the proof of our main theorem.


\begin{remark} 
This proof has a cohomological nature.
In step 2, when proving fully faithfulness, we use the behavior of the map over the kernel to build a cocycle
$\beta\in C^1(\u K,E\to C')$ in the transformation complex, and in step 3, when proving essential surjectivity, we build a cocycle $\gamma\in C^2(\u K,E\to C)$ measuring the failure of the cleavage to be $\U$-flat.
We show that the cohomology groups $H^1(\u K,E\to C')$ and 
$H^2(\u K,E\to C)=0$ vanish, by building cochain integration to our cocycles, with the aid of a partition of 1.
\end{remark}


\frenchspacing

\footnotesize{

}

\bigskip

\sf{\noindent Matias del Hoyo\\
Universidade Federal Fluminense (UFF), Departamento de Geometria (GGM).\\ 
Rua Prof. Marcos W. de Freitas Reis s/n, campus Gragoat\'a G, 24210-201 Niterói, RJ, Brasil.\\
mldelhoyo@id.uff.br}

\

\sf{\noindent Cristian Oritz\\
Universidade de S\~ao Paulo (USP), Instituto de Matem\'atica e Estat\'istica (IME).\\
Rua do Mat\~ao 1010, 05508-090, S\~ao Paulo, SP, Brasil.\\
cortiz@ime.usp.br}

\end{document}